\documentclass[11pt]{article}
\usepackage{authblk}
\usepackage[T1]{fontenc}
\usepackage{lmodern}
\usepackage[margin=1.05in]{geometry}
\usepackage{amsmath,amssymb,amsthm,mathtools}
\usepackage{microtype,booktabs,enumitem}
\usepackage{xcolor,fancyhdr}
\usepackage[colorlinks=true,linkcolor=blue!45!black,citecolor=blue!45!black,urlcolor=blue!45!black]{hyperref}

\numberwithin{equation}{section}

\newtheorem{theorem}{Theorem}[section]
\newtheorem{proposition}[theorem]{Proposition}
\newtheorem{lemma}[theorem]{Lemma}
\newtheorem{corollary}[theorem]{Corollary}

\theoremstyle{definition}
\newtheorem{definition}[theorem]{Definition}
\theoremstyle{remark}

\newcommand{\N}{\mathbb N}
\newcommand{\Z}{\mathbb Z}
\newcommand{\Q}{\mathbb Q}
\newcommand{\R}{\mathbb R}
\newcommand{\C}{\mathbb C}
\newcommand{\T}{\mathbb T}
\newcommand{\dd}{\,\mathrm d}

\newcommand{\rank}{\operatorname{rank}}
\newcommand{\Crit}{\operatorname{Crit}}
\newcommand{\B}{\mathcal B}
\newcommand{\M}{\mathfrak M}
\newcommand{\m}{\mathfrak m}
\newcommand{\SingSeries}{\mathfrak S}
\newcommand{\JJ}{\mathfrak J}
\newcommand{\av}[1]{\int_X #1\dd\nu}
\newcommand{\norm}[1]{\lVert #1\rVert}

\title{\bfseries Dynamical Chowla conjecture for polynomials of several variables}

\author[1]{Biao Wang}
\affil[1]{School of Mathematics and Statistics, Yunnan University, Kunming, Yunnan 650500, China\\ \texttt{bwang@ynu.edu.cn}}

\author[2]{Daodao Yang\footnote{Corresponding author: Daodao Yang.}}
\affil[2]{School of Mathematical Science, Eastern Institute of Technology, Ningbo, Zhejiang 315200, China\\ \texttt{yangdao2@126.com}}

\date{\today}

\begin{document}
\maketitle
\thispagestyle{plain}
\vspace{-1.5em}
\begin{abstract}
In 1965, Chowla conjectured that the mean value of the Liouville function at non-square one-variable polynomials is zero. This conjecture is a generalization of the prime number theorem (PNT). In 2022, Bergelson and Richter established a dynamical generalization of this form of the PNT. For non-square multivariate polynomials, one  also expects that Chowla's conjecture is true. In this article, we extend Bergelson and Richter's result to multivariate polynomials whose Birch rank is sufficiently large in terms of their degree. We also establish analogous results for averages over integer points on algebraic varieties.
\end{abstract}



\section{Introduction}
Let $\Omega(n)$ denote the number of prime factors of a positive integer
$n$, counted with multiplicity, with $\Omega(1)= \Omega(0)= 0$ and $\Omega(t)=\Omega(|t|)$ for any integer $t<0$. Put $\lambda(t)=(-1)^{\Omega(t)}$ for $t\in\Z$, including the convention $\lambda(0)=1$. A nonconstant polynomial $P(x)\in \mathbb{Z}[x]$ is \emph{non-square} if $P(x)$ has positive leading coefficient and $P(x)$ is not of the form $cQ(x)^2$ for any integer $c\ge1$ and polynomial $Q(x)\in \mathbb{Z}[x]$. In 1965, Chowla~\cite{Chowla1965} conjectured that 
\begin{align}\label{chowla_conj}
\lim_{N\to\infty} \frac1N\sum_{n\leq N}\lambda(P(n))=0
\end{align}
for any such polynomial. For $P(n)=n$, it is an equivalent form of the prime number theorem (PNT). For $P(n)=qn+a$ for any integers $q\ge1$ and $0\leq a\leq q-1$, it is an equivalent form of the PNT in arithmetic progressions.  Let $k\ge2$. Equation \eqref{chowla_conj} also contains Chowla's well-known conjecture on $k$-point correlations. See \cite{MRT2015, Tao2016, TT2019, HelfgottRadziwill2021, Pilatte2026, OAI} for the developments in this direction.  A nonconstant multivariate polynomial $P\in \Z[X_1,\dots,X_k]$ is \emph{non-square} if $P$ is not of the form $cQ^2$ for any $c\in\Z\setminus\{0\}$ and $Q\in \mathbb{Z}[X_1,\dots,X_k]$. By analogy with Chowla's conjecture for univariate polynomials, one may conjecture that
\begin{align}\label{chowla_conj_multivariate}
\lim_{N\to\infty} \frac1{N^k}\sum_{n_1,\dots, n_k\leq N}\lambda(|P(n_1,\dots, n_k)|)=0
\end{align}
for any non-square polynomial. For example, Helfgott \cite{Helfgott_1, Helfgott_2} proved that \eqref{chowla_conj_multivariate} holds for cubic homoegeneous polynomial. In \cite{Teravainen}, Ter\"av\"ainen proved that there are infinitely many coprime pairs $(m,n)$ such that $\lambda(am^3+bn^4)=\sigma$ for $\sigma\in \{-1,1\}, 1\leq a,b\leq 100, k\ge1$ and the same for $\lambda(am^2+bn^k)=\sigma$.

In 2022, Bergelson and Richter
\cite{BR} established a dynamical generalization of \eqref{chowla_conj} for  $P(n)=n$.  Let  $(X,\nu,T)$ be a uniquely ergodic topological dynamical system, where $X$ is a compact metric space, $T:X\to X$ is
continuous, and $\nu$ is the unique
$T$-invariant Borel probability measure. 
Bergelson and Richter
\cite{BR} proved that
\begin{equation}\label{eq:BRintro}
 \lim_{N\to\infty}\frac1N\sum_{n=1}^N f(T^{\Omega(n)}x)=\av f
\end{equation}
for every $x\in X$, and every $f\in C(X)$. Their arithmetic progression
version, \cite[Corollary 1.16]{BR}, permits $n$ to be replaced by
$qn+a$ for any fixed modulus $q\ge1$ and $a\in\{0,1,\dots,q-1\}$. Their work has stimulated a number of generalizations and
variants of \eqref{eq:BRintro};  see \cite{Loyd2023, Xiao2025, DLMS, CharamarasRichter2025, Wang2025procA, WangYi2026jnt, CespedesDonoso2026, Araujo2026,BergelsonReillyRichter2026,QZ}. In \cite{WangYi2026}, Wang and Yi conjectured that the analogous dynamical generalization of the multivariate Chowla conejcture \eqref{chowla_conj_multivariate} holds:  let $P\in \Z[X_1,\dots,X_k]$ be a polynomial with coprime coefficients such that there is no integer $d\ge2$ such that the values of $|P(X_1,\dots,X_k)|$ at almost all positive integers are $d$th powers, then
	\begin{equation}\label{conj_multi_poly_eqn}
		\lim_{N\to\infty} \frac1{N^k}\sum_{1\leq n_1,\dots, n_k\leq N} f(T^{\Omega(|P(n_1,\dots,n_k)|)}x) = \int_X f\,d\nu
	\end{equation}
holds for any $x\in X$ and $f\in C(X)$. In \cite{QZ}, Qi and Zheng gave a criterion that implies \eqref{conj_multi_poly_eqn}. In this article, we will show that if the Birch rank of the highest-degree homogeneous part is sufficiently large in terms of the degree, then \eqref{conj_multi_poly_eqn} holds. Moreover, we will generalize it to averages over integer points in algebraic varieties. 

\subsection{Main results for polynomials of several variables}

Let $k\ge1$ be an integer.
If $H\in\Z[X_1,\ldots,X_k]$ is a nonzero homogeneous polynomial of
degree ${d}\ge2$, define its \emph{Birch rank} by
\begin{equation}\label{eq:Birchdef}
 \B(H)=k-\dim_{\C}\Crit(H),\qquad
 \Crit(H)=\{\mathbf n\in\C^k:\nabla H(\mathbf n)=0\}.
\end{equation}
For example, a smooth
projective hypersurface has $\B(H)=k$. For a polynomial $P$ of degree
${d}$, write $P_{d}$ for its highest-degree homogeneous part and set
$\B(P)=\B(P_{d})$. Put $[N]=\{1,\ldots,N\}$. 

\begin{theorem}
\label{thm:Birch}
Let  $(X,\nu,T)$ be a uniquely ergodic system. Let $P\in\Z[X_1,\ldots,X_k]$ be a fixed polynomial of degree
${d}\ge2$.
If the Birch rank of $P$ satisfies
\begin{equation}\label{eq:rankcondition}
 \B(P)>({d}-1)2^{{d}-1},
\end{equation}
then we have
\begin{equation}\label{eq:generalmain}
 \lim_{N\to\infty}\frac1{N^k}
 \sum_{\mathbf n\in[N]^k}f\bigl(T^{\Omega(P(\mathbf n))}x\bigr)=\av f
\end{equation}
 for every $x\in X$ and
every $f\in C(X)$.
\end{theorem}

The  condition \eqref{eq:rankcondition} is from 
Birch's classical work  \cite{Birch} on forms in many variables. One may improve  \eqref{eq:rankcondition}  by using refinements of Birch's theorem, such as those in \cite{DW2026}. Theorem~\ref{thm:Birch} gives the following two concrete consequences for  quadratic polynomials.

\begin{corollary}\label{cor:quadratic}
Let  $(X,\nu,T)$ be a uniquely ergodic system. Let $Q\in\Z[X_1,\ldots,X_k]$ be a quadratic polynomial, written as
\[
 Q(\mathbf X)=\mathbf X^{\mathsf T}A\mathbf X+L(\mathbf X)+c,
\]
where $A$ is a symmetric rational matrix, $L$ is linear, and $c$ is
constant. If $\rank A\ge3$, then for every $x\in X$ and every $f\in C(X)$, we have
\begin{equation}\label{eq:quadraticmain}
 \lim_{N\to\infty}\frac1{N^k}
 \sum_{\mathbf n\in[N]^k}
 f\bigl(T^{\Omega(Q(\mathbf n))}x\bigr)=\av f.
\end{equation}

\end{corollary}
Indeed, the leading form is $H(\mathbf X)=\mathbf X^{\mathsf T}A\mathbf X$.
Since $\nabla H=2A\mathbf X$,
\[
 \Crit(H)=\ker_{\C}A,\qquad \B(H)=\rank A>2.
\]
Theorem~\ref{thm:Birch} with ${d}=2$ proves Corollary~\ref{cor:quadratic}.  

It is a well-known conjecture in analytic number theory that there are infinitely many positive integers $n$ such that $n^2+1$ is prime. Chowla's conejcture states that 
$$\lim_{N\to\infty} \frac1N\sum_{n\leq N}\lambda(n^2+1)=0,
$$
which remains open. As a consequence of Corollary~\ref{cor:quadratic}, we prove that this holds for the sum of three squares plus one. 

\begin{corollary}\label{cor:squaresplusone}
 Let $k\ge3$. For every uniquely ergodic system $(X,\nu,T)$, every
$x\in X$, and every $f\in C(X)$, we have
\begin{equation}\label{eq:squaresplusone}
 \lim_{N\to\infty}\frac1{N^k}
 \sum_{n_1,\ldots,n_k=1}^N
 f\!\left(T^{\Omega(n_1^2+\cdots+n_k^2+1)}x\right)=\av f.
\end{equation}
In particular,
\begin{equation}
\lim_{N\to\infty} \frac1{N^3}\sum_{n_1,n_2,n_3=1}^N\lambda(n_1^2+n_2^2+n_3^2+1)=0
\end{equation}
\end{corollary}

The lower bound in \eqref{eq:rankcondition} grows exponentially
in $d$. For signed sums of powers, we can reduce this
exponential bound to a bound of quadratic order in $d$. Fix signs
$\boldsymbol\varepsilon=(\varepsilon_1,\ldots,\varepsilon_k)
\in\{1,-1\}^k$, and let the signed sum of  $n_j^{d}$ be the polynomial defined by
\begin{equation}\label{eq:signedform}
 P_{\boldsymbol\varepsilon}(\mathbf n)
 =\sum_{j=1}^k\varepsilon_j n_j^{d},
\end{equation}

\begin{theorem}\label{thm:diagonal}
Let  $(X,\nu,T)$ be a uniquely ergodic system. Let ${d}\ge2$ be an integer. For ${d}\ge4$, put
\[
 D_{d}={d}({d}+1),\qquad a_j=\min\{2^j,j(j+1)\},
 \qquad
 \Theta_{d}=\min_{1\le j\le{d}-1}
 \left\{D_{d}-\frac{D_{d}-a_j}{{d}+1-j}\right\}.
\]
Define
\begin{equation}\label{eq:K}
 K({d})=
 \begin{cases}
  2,&{d}=2,\\
  4,&{d}=3,\\
  \left\lfloor\Theta_{d}/2\right\rfloor+1,&{d}\ge4.
 \end{cases}
\end{equation}
Let $k\ge K({d})$ and let
$\boldsymbol\varepsilon\in\{1,-1\}^k$ be arbitrary. Then for every $x\in X$ and every $f\in C(X)$, we have
\begin{equation}\label{eq:diagonalmain}
 \lim_{N\to\infty}\frac1{N^k}
 \sum_{\mathbf n\in[N]^k}
 f\bigl(T^{\Omega(P_{\boldsymbol\varepsilon}(\mathbf n))}x\bigr)=\av f.
\end{equation}
\end{theorem}

The sufficient bounds in Theorem~\ref{thm:diagonal}  are
\begin{center}
\begin{tabular}{c*{11}{r}}
\toprule
${d}$ & $2$ & $3$ & $4$ & $5$ & $6$ & $7$ & $8$ & $9$ & $10$ & $11$ & $12$\\
\midrule
$K({d})$ & $2$ & $4$ & $8$ & $12$ & $17$ & $24$ & $31$ & $39$ & $49$ & $59$ & $70$\\
\bottomrule
\end{tabular}
\end{center}
Moreover, Proposition~\ref{prop:thresholdcomparison} shows that, for
${d}\ge4$,
\[
 K({d})\le \frac{{d}({d}-1)}2+
 \left\lfloor\sqrt{2{d}+2}\right\rfloor =\frac{d^2}{2}-\frac d2+\sqrt{2d}+O(1),
\]
which is of quadratic growth. 

In \cite{QZ}, Qi and Zheng asked whether \eqref{conj_multi_poly_eqn} holds for $P(n_1,n_2)=n_1^3+n_2^3$.
Allowing more variables, Theorem~\ref{thm:diagonal} shows that
\eqref{conj_multi_poly_eqn} holds for the sum of four cubes.

\begin{corollary}\label{cor:cubic}
  For every uniquely ergodic system $(X,\nu,T)$, every
$x\in X$, and every $f\in C(X)$, we have
\begin{equation}\label{eq:cubic}
 \lim_{N\to\infty}\frac1{N^4}
 \sum_{n_1,n_2,n_3,n_4=1}^N
 f\!\left(T^{\Omega(n_1^3+n_2^3+n_3^3+n_4^3)}x\right)=\av f.
\end{equation}
\end{corollary}

\subsection{Generalizations to  integer points in algebraic varieties}

For a tuple $\mathbf A=(A_1,\ldots,A_s)$ of homogeneous forms of the
same degree in $m$ variables, put
\begin{equation}\label{var:eq:system-rank}
 \Crit(\mathbf A)=
 \{\mathbf z\in\C^m:\rank J_{\mathbf A}(\mathbf z)<s\},
 \qquad
 \B(\mathbf A)=m-\dim_{\C}\Crit(\mathbf A).
\end{equation}
Here $J_{\mathbf A}$ is the matrix with rows $\nabla A_i$.
We set $\B(\varnothing)=+\infty$. Write
\[
 C_d=(d-1)2^{d-1},\qquad K_{r,d}=(2r+1)(2r+2)C_d.
\]
All boxes below are fixed bounded boxes with sides parallel to the
coordinate axes and with nonempty interior. 

\begin{theorem}
\label{var:thm:main}
Let $F_1,\ldots,F_r,P\in\Z[x_1,\ldots,x_n]$ have common degree
$d\ge2$, and write
\[
 \mathbf F=(F_1,\ldots,F_r),\qquad
 \mathbf G=((F_1)_d,\ldots,(F_r)_d),\qquad H=P_d.
\]
Suppose
\begin{equation}\label{var:eq:main-rank}
 \min\{\B(\mathbf G),\,2\B(\mathbf G,H)\}>K_{r,d}.
\end{equation}
Let $\mathcal R\subset\R^n$ be a box. Assume that $\mathbf G=0$
has a nonsingular real solution in $\mathcal R^\circ$, and that
$\mathbf F=0$ has a nonsingular solution in $\Z_p^n$ for every prime
$p$.  Set
\[
 V_N=\{\mathbf x\in N\mathcal R\cap\Z^n:\mathbf F(\mathbf x)=0\}.
\]
Then, for every uniquely ergodic system $(X,\nu,T)$, every $x\in X$,
and every $f\in C(X)$,
\begin{equation}\label{var:eq:main-limit}
 \lim_{N\to\infty}\frac{1}{|V_N|}
 \sum_{\mathbf z\in V_N}f(T^{\Omega(P(\mathbf z))}x)
 =\av f.
\end{equation}
\end{theorem}

Under these hypotheses, Birch's theorem and its polynomial version give
\begin{equation}\label{var:eq:base-count}
 |V_N|=c_VN^h+O(N^{h-\delta}),\qquad
 h=n-rd,\quad c_V>0,\quad\delta>0.
\end{equation}
See \cite[Theorem~1.2]{BHB} and the paragraph
following it. For $r=0$, $c_V$ is the volume of $\mathcal R$. So for $r=0$ and $\mathcal R=(0,1]^n$,
\eqref{var:eq:main-rank} is $2\B(H)>2C_d$, 
Theorem~\ref{var:thm:main} recovers Theorem~\ref{thm:Birch} with
exactly the same rank bound \eqref{eq:rankcondition}. For a concrete example, one may take $r=1$, $F=x_1^2+\cdots+x_{24}^2-x_{25}^2$, and $P=\sum_{j=1}^{25} jx_j^2$. For $r=2$, let $m\ge31$ and  put
\[
Q_i(\mathbf x)=\sum_{j=1}^{m}x_{(i-1)m+j}^{2}, \quad i=1,2,3.
\]
Define
\[
F_1=Q_1+Q_2-2Q_3,\qquad
F_2=Q_1-2Q_2+Q_3,\qquad
P=Q_1+Q_2+Q_3.
\]
The equations $F_1=F_2=0$ are equivalent to $Q_1=Q_2=Q_3$.
Therefore, setting
\[
V_N=
\{\mathbf z\in[N]^{3m}:Q_1(\mathbf z)=Q_2(\mathbf z)=Q_3(\mathbf z)\},
\]
Theorem~\ref{var:thm:main} gives
\[
\lim_{N\to\infty}\frac1{|V_N|}
\sum_{\mathbf z\in V_N}
f\!\left(T^{\Omega(z_1^2+\cdots+z_{3m}^2)}x\right)
=
\int_X f\,d\nu
\]
for every $x\in X$ and every $f\in C(X)$. 

We now extend Theorem~\ref{thm:diagonal} to integer points constrained
by homogeneous equations.  Let $K(d)$ be defined by \eqref{eq:K}, and put
\begin{equation}\label{diagvar:eq:kstar}
 K_*(d)=\begin{cases}3,&d=2,\\ K(d),&d\ge3.\end{cases}
\end{equation}

\begin{theorem}
\label{diagvar:thm:main}
Let $r\ge1$, $d\ge2$, $A=(a_{ij})\in\Z^{r\times n}$, and
$\boldsymbol\varepsilon\in\{1,-1\}^n$. Define
\[
 F_i(\mathbf z)=\sum_{j=1}^n a_{ij}z_j^d\quad(1\le i\le r),
 \qquad P(\mathbf z)=\sum_{j=1}^n\varepsilon_j z_j^d,
 \qquad B=\begin{pmatrix}A\\\boldsymbol\varepsilon\end{pmatrix}.
\]
Suppose that every $r$ columns of $A$ are linearly independent, and
every $r+1$ columns of $B$ are linearly independent. Assume
\begin{equation}\label{diagvar:eq:variable-bound}
 n\ge(2r+1)K_*(d).
\end{equation}
Suppose also that $\mathbf F=0$ has a nonsingular real solution in
$(0,1)^n$ and a nonsingular solution in $\Z_p^n$ for every prime $p$.
Set
\[
 V_N=\{\mathbf z\in[N]^n:F_1(\mathbf z)=\cdots=F_r(\mathbf z)=0\}.
\]
Then, for every uniquely ergodic system $(X,\nu,T)$, every $x\in X$,
and every $f\in C(X)$,
\begin{equation}\label{diagvar:eq:main-limit}
 \lim_{N\to\infty}\frac1{|V_N|}
 \sum_{\mathbf z\in V_N}f(T^{\Omega(P(\mathbf z))}x)=\av f.
\end{equation}
\end{theorem}

The column hypotheses have a geometric formulation: the projective
diagonal varieties $V(\mathbf F)$ and $V(\mathbf F,P)$ are smooth
complete intersections over $\Q$ (equivalently, after base change to
$\C$) of codimensions $r$ and $r+1$, respectively. For a single diagonal equation, \eqref{diagvar:eq:variable-bound}
gives the following sufficient bounds:
\begin{center}
\begin{tabular}{c*{5}{r}}
\toprule
$d$&2&3&4&5&6\\
\midrule
$n\ge3K_*(d)$&9&12&24&36&51\\
\bottomrule
\end{tabular}
\end{center}

For a concrete  example, let $d=3$, $n=12$, and put
\[
 F(\mathbf z)=\sum_{j=1}^{12}(2j-13)z_j^3,
 \qquad P(\mathbf z)=\sum_{j=1}^{12}z_j^3,
\]
\[
V_N=\{\mathbf z\in[N]^{12}:F(\mathbf z)=0\}.
\]
Then $|V_N|\sim c_VN^9$ with $c_V>0$, and Theorem~\ref{diagvar:thm:main}  gives
\[
\lim_{N\to\infty} \frac1{|V_N|}\sum_{\mathbf z\in V_N}
 f\bigl(T^{\Omega(z_1^3+\cdots+z_{12}^3)}x\bigr)
 =\av f.
\]

\subsection{Main ideas in the proofs}

To prove the main theorems, we use the circle method. For a multivariate polynomial $P\in\Z[X_1,\dots,X_k]$,
define the associated exponential sum by
 \begin{equation*}
 S_{P,N}(\alpha)=\sum_{\mathbf n\in[N]^k}e(\alpha P(\mathbf n)),
 \qquad e(t)=\exp(2\pi it).
\end{equation*}

Fix an integer $C\ge1$ such that $|P(\mathbf n)|\le CN^{d}$ on
$[N]^k$, put $L_N=CN^{d}$. For any bounded sequence $b:\Z\to\C$, define
\begin{equation*}
 B_N(\alpha)=\sum_{|t|\le L_N}b(t)e(-t\alpha).
\end{equation*}
Let $\T=\R/\Z$. All integrals over $\T$ are taken with respect
to normalized Lebesgue measure. Then
\begin{equation*}
\sum_{\mathbf n\in[N]^k}b(P(\mathbf n)) = \int_{\mathbb T}S_{P,N}(\alpha)B_N(\alpha)\,d\alpha
\end{equation*}
For example,  for any fixed $x\in X$ and $f\in C(X)$, if we take 
\begin{equation*}
 b(n)=b_x(n)=f(T^{\Omega(n)}x)-\av f\qquad(n\in\Z),
\end{equation*}
then to prove \eqref{eq:generalmain}, it suffices to prove that
\begin{equation*}
 \lim_{N\to\infty}\frac1{N^k}\sum_{\mathbf n\in[N]^k}b(P(\mathbf n))=0,
\end{equation*}
which is equivalent to proving
\begin{equation*}
 \lim_{N\to\infty}N^{-k}\int_{\mathbb T}S_{P,N}(\alpha)B_N(\alpha)\,d\alpha=0.
\end{equation*}

If $b$ has vanishing rational means, then its cancellation controls the major-arc contribution. For the minor arcs, using the Cauchy-Schwarz inequality yields an $L^2$ integral of $S_{P,N}(\alpha)$. Using the identity
\[
 \int_{\T}|S_{P,N}(\alpha)|^2\dd\alpha=\#\{(\mathbf u,\mathbf v)\in[N]^{2k}:
                      P(\mathbf u)=P(\mathbf v)\},
\]
Birch's theorem gives an asymptotic formula for this integral. This will imply that the contribution from the minor arcs is negligible. 

For $\boldsymbol\varepsilon=(\varepsilon_1,\ldots,\varepsilon_k)
\in\{1,-1\}^k$,  we have 
\[
 |S_{P_{\boldsymbol\varepsilon},N}|^2=|W_{{d},N}|^{2k},
\quad \text{where}\quad
 W_{d,N}(\alpha)=\sum_{n=1}^N e(\alpha n^d).
\]
Therefore, we can use the classical estimate on the moments of the Weyl sum to control the square integral of $S_{P_{\boldsymbol\varepsilon},N}$.

In \cite{QZ}, Qi and Zheng use  the exponential sum along $\Omega(|P(\mathbf n)|)$ to give a criterion that implies \eqref{eq:generalmain}. In this article, we use the exponential sum along $P(\mathbf n)$. The arguments based on Proposition~\ref{prop:transfer}, including their constrained analogues, also apply to arithmetic functions $a:\N\to\N_0$, where $\N_0=\{0,1,2,\ldots\}$, satisfying
\begin{equation*}
 \lim_{N\to\infty}\frac1N\sum_{n=1}^N f(T^{a(qn+r)}x)=\av f
\end{equation*}
for every $q\ge1$, $r\in\{0,1,\ldots,q-1\}$, $x\in X$, and $f\in C(X)$, see Lemma~\ref{lem_ap}. For integer polynomial values, extend $a$ by $a(-n)=a(n)$ and choose any $a(0)\in\N_0$. This gives the analogous conclusions of Theorems~\ref{thm:Birch}, \ref{var:thm:main} and \ref{diagvar:thm:main}, and of Theorem~\ref{thm:diagonal} except for the separately treated case $d=k=2$. The progression hypothesis alone does not imply those two-variable quadratic cases.

For example, for every integer base $b\ge2$, \eqref{eq:generalmain}
still holds if $\Omega(P(\mathbf n))$ is replaced by
$s_b(\Omega(P(\mathbf n)))$, where $s_b(j)$ denotes the sum of the
base-$b$ digits of $j\ge0$, with $s_b(0)=0$.
The required progression averages are proved in
Corollary~\ref{cor:digit-omega-ap} below. In \cite[Theorem 1.9]{WangYi2026}, Wang and Yi proved the case $q=1$.

In Section~\ref{sec:transfer}, using the ideas in this subsection,  we will give a new  criterion that implies \eqref{eq:generalmain}. Then we will apply it to prove the main theorems one by one in the remaining sections.

\section{Circle method}\label{sec:transfer}

Let  $(X,\nu,T)$ be a uniquely ergodic system. Let $e(u)=\exp(2\pi i u)$ for any $u\in\R$. Let $\T=\R/\Z$, and all integrals over $\T$ use the normalized Lebesgue measure.

\begin{definition}\label{def:null}
A bounded sequence $b:\Z\to\C$ has \emph{vanishing rational means}
if, for every $\theta\in\Q$ and $\sigma\in\{1,-1\}$,
\begin{equation}\label{eq:null}
\lim_{N\to\infty}\frac1N \sum_{n=1}^N b(\sigma n)e(n\theta)=0.
\end{equation}
\end{definition}

\begin{lemma}\label{lem_ap}
Let $b:\Z\to\C$ be bounded. If 
\begin{equation}\label{lem_ap_eq}
 \lim_{N\to\infty}\frac1N\sum_{n=1}^{N}b(qn+a)=0
\end{equation}
 for any nonzero integer $q$ and  every $a\in\{0,\ldots,|q|-1\}$, then $b(t)$ has vanishing rational means.
\end{lemma}
\begin{proof}
It suffices to show that
\begin{equation}\label{bx_null}
\lim_{N\to\infty}\frac1N \sum_{n=1}^N b( n)e(n\theta)=0
\end{equation}
for every  $\theta\in\mathbb Q$. The case $b(-n)$ is similar.
  
Write $\theta=r/q$, where $r\in\mathbb Z$ and
$q\ge1$. For each residue class $a\bmod q$, define
\[
 A_a(N)=
 \sum_{\substack{1\le t\le N\\t\equiv a\pmod q}}b(t).
\]
Then by \eqref{lem_ap_eq}, we have
\[
 \lim_{N\to\infty} \frac{A_a(N)}N =0.
\]

If $t\equiv a\pmod q$, then
$e(rt/q)=e(ra/q)$. Splitting the sum into residue classes 
gives
\[
 \sum_{t=1}^{N}b(t)e(rt/q)
 =
 \sum_{a=0}^{q-1}e(ra/q)A_a(N).
\]
Since $q$ is fixed, it follows that
\[
 \left|
 \frac1N\sum_{t=1}^{N}b(t)e(rt/q)
 \right|
 \le
 \sum_{a=0}^{q-1}\frac{|A_a(N)|}{N}
 \longrightarrow0.
\]
This proves \eqref{bx_null}.
\end{proof}

By \cite[Corollary 1.16]{BR}, we get the following corollary.
\begin{corollary}\label{lem:arithmetic}
For a fixed $x\in X$ and $f\in C(X)$,
the sequence
\begin{equation}\label{eq:bx}
 b_x(n)=f(T^{\Omega(n)}x)-\av f\qquad(n\in\Z)
\end{equation}
has vanishing rational means.
\end{corollary}

By the Schwartz--Zippel lemma
\cite[Corollary 1]{Schwartz}, we have the following bound on the number of zeros of multivariable polynomials.

\begin{lemma}\label{lem:zeros}
If $P\in\Z[X_1,\ldots,X_k]$ is nonzero, then
\[
 \#\{\mathbf n\in[N]^k:P(\mathbf n)=0\}=O_P(N^{k-1}).
\]
\end{lemma}

Define $\| t \| : = \min\{|n-t|: n\in\Z\}$. For a fixed degree ${d}\ge2$ and integer $Q\ge1$, define
\begin{align}
 \M_N(Q)&=\bigcup_{\substack{1\le q\le Q\\0\le a<q,\ (a,q)=1}}
 \left\{\alpha\in\T:\left\|\alpha-\frac aq\right\|\le QN^{-{d}}\right\},
 \label{eq:fixedarcs}\\
 \m_N(Q)&=\T\setminus\M_N(Q).\notag
\end{align}
We have
\begin{equation}\label{eq:measure}
 |\M_N(Q)|\le 2Q^3N^{-{d}}.
\end{equation}

\begin{proposition}\label{prop:transfer}
Let $P\in\Z[X_1,\ldots,X_k]$ have degree ${d}\ge2$. Let
\begin{equation}\label{eq:SPintro}
 S_{P,N}(\alpha)=\sum_{\mathbf n\in[N]^k}e(\alpha P(\mathbf n)).
 \end{equation}
Suppose that
\begin{equation}\label{eq:energy}
 \lim_{Q\to\infty}\limsup_{N\to\infty}
 N^{{d}-2k}\int_{\m_N(Q)}|S_{P,N}(\alpha)|^2\dd\alpha=0.
\end{equation}
If $b$ is bounded and has vanishing rational means, then
\begin{equation}\label{eq:transferresult}
 \lim_{N\to\infty}\frac1{N^k}\sum_{\mathbf n\in[N]^k}b(P(\mathbf n))=0.
\end{equation}
Consequently, \eqref{eq:generalmain} holds for $P$.
\end{proposition}
\begin{proof}
Fix an integer $C\ge1$ such that $|P(\mathbf n)|\le CN^{d}$ on
$[N]^k$, put $L_N=CN^{d}$, and define
\begin{equation}\label{eq:BN}
 B_N(\alpha)=\sum_{|t|\le L_N}b(t)e(-t\alpha).
\end{equation}
Then for $|t|\le L_N$ we have 
\[
b(t)=\int_{\mathbb T}B_N(\alpha)e(t\alpha)\,d\alpha.
\]
By Parseval's formula, we have
\begin{align}
\int_{\mathbb T}|B_N(\alpha)|^2\,d\alpha =\sum_{|t|\le L_N}|b(t)|^2
 \ll_C\norm b_\infty^2N^{d}.\label{eq:Parseval}
\end{align}
Moreover, 
\begin{align}
\int_{\mathbb T}S_{P,N}(\alpha)B_N(\alpha)\,d\alpha
&=
\sum_{\mathbf n\in[N]^k}
\sum_{|t|\le L_N}
b(t)\int_{\mathbb T}
e\bigl(\alpha(P(\mathbf n)-t)\bigr)\,d\alpha \nonumber\\
&=
\sum_{\mathbf n\in[N]^k}
\sum_{|t|\le L_N}
b(t)\,\mathbf1_{\{t=P(\mathbf n)\}} \nonumber\\
&=
\sum_{\mathbf n\in[N]^k}b(P(\mathbf n)).  \label{eq:Fourier}
\end{align}

For every fixed rational $a/q$ and either sign, \eqref{eq:null} implies
\begin{equation}\label{eq:maximalpartial}
 \max_{M\le L_N}\left|\sum_{t=1}^M b(\pm t)e(\mp at/q)\right|
 =o(N^{d}).
\end{equation}
Then partial summation shows that
\[
 \sup_{|\beta|\le QN^{-{d}}}
 |B_N(a/q+\beta)|=o_Q(N^{d}).
\]
There are only finitely many rationals in \eqref{eq:fixedarcs}. Hence
\begin{equation}\label{eq:Bmajor}
 \sup_{\alpha\in\M_N(Q)}|B_N(\alpha)|=o_Q(N^{d}).
\end{equation}

We now split \eqref{eq:Fourier} over $\M_N(Q)$ and $\m_N(Q)$.
The elementary bound $|S_{P,N}(\alpha)|\le N^k$ and
\eqref{eq:measure} show that
\begin{equation}\label{eq:majortransfer}
 N^{-k}\left|\int_{\M_N(Q)}S_{P,N}B_N\dd\alpha\right|
 \le |\M_N(Q)|\sup_{\M_N(Q)}|B_N|
 \le 2Q^3N^{-{d}}\cdot o_Q(N^{d})=o_Q(1).
\end{equation}

On the minor arcs, Cauchy--Schwarz gives
\[
 N^{-k}\left|\int_{\m_N(Q)}S_{P,N}B_N\dd\alpha\right|
 \le N^{-k}
 \left(\int_{\m_N(Q)}|S_{P,N}|^2\dd\alpha\right)^{1/2}
 \|B_N\|_2.
\]
Substituting \eqref{eq:Parseval} 
yields
\begin{equation}\label{eq:minortransfer}
 N^{-k}\left|\int_{\m_N(Q)}S_{P,N}B_N \dd\alpha\right|
 \ll_C\norm b_\infty
 \left(N^{{d}-2k}\int_{\m_N(Q)}|S_{P,N}|^2\dd\alpha\right)^{1/2}.
\end{equation}

For clarity, put
\[
 E(Q)=\limsup_{N\to\infty}
 N^{{d}-2k}\int_{\m_N(Q)}|S_{P,N}(\alpha)|^2\dd\alpha.
\]
Taking the upper limit in $N$ with $Q$ fixed in \eqref{eq:majortransfer} and \eqref{eq:minortransfer} gives
\[
 \limsup_{N\to\infty}
 \left|N^{-k}\sum_{\mathbf n\in[N]^k}b(P(\mathbf n))\right|
 \ll_C\|b\|_\infty E(Q)^{1/2}.
\]
Hypothesis \eqref{eq:energy} states that $E(Q)\to0$ as
$Q\to\infty$. This proves \eqref{eq:transferresult} with the
required order of limits.

For the dynamical conclusion take $b=b_x$ from
Corollary~\ref{lem:arithmetic}. Then $b_x$ has vanishing rational means.  Its normalized polynomial average is
exactly the left side of \eqref{eq:generalmain} minus $\int_Xf\dd\nu$. Lemma~\ref{lem:zeros} shows that the
contribution of the value $b_x(0)$ in the  average of  \eqref{eq:generalmain}  is $O_P(N^{-1})$. 
Thus,   \eqref{eq:generalmain}   follows by \eqref{eq:transferresult}.
\end{proof}

\section{Proof of Theorem~\ref{thm:Birch}}\label{sec:Birch}

The following is the single-equation case of
\cite[Theorem 1.2]{BHB}, together with the discussion immediately
following that theorem concerning nonhomogeneous polynomials.
The Birch rank condition is imposed on the highest-degree homogeneous
part. In the homogeneous case, this is the classical theorem
of Birch \cite{Birch}; see also \cite[Theorem 1.1]{Schindler} with one equation.

\begin{proposition}\label{input:Birch}
Let $F\in\Z[X_1,\ldots,X_k]$ have degree ${d}\ge2$. If
\begin{equation}\label{eq:Birchclassicalrank}
\B(F)>({d}-1)2^{d},
\end{equation}
then there exists some $\delta>0$ such that
\begin{equation}\label{eq:Birchcount}
 \#\{\mathbf n\in[N]^k:F(\mathbf n)=0\}
 =\SingSeries_F\JJ_{F_{d}}N^{k-{d}}+O_F(N^{k-{d}-\delta}).
\end{equation}
Here the singular series and singular integral are absolutely convergent,
and are given by
\begin{align}
 \SingSeries_F&=\sum_{q=1}^{\infty}\ 
 \sum_{\substack{0\le a<q\\(a,q)=1}}
 q^{-k}\sum_{\mathbf n\bmod q}e(aF(\mathbf n)/q),
 \label{eq:SSgeneral}\\
 \JJ_{F_{d}}&=\int_{\R}\int_{[0,1]^k}
 e(\gamma F_{d}(\mathbf u))\dd\mathbf u\dd\gamma.
 \label{eq:JJgeneral}
\end{align}
\end{proposition}

Let $P\in\Z[X_1,\ldots,X_k]$ be a polynomial of degree
${d}\ge2$
satisfying  \eqref{eq:rankcondition}. That is,  $\B(P)>({d}-1)2^{{d}-1}.$
Set $H=P_{d}$. Consider
\begin{equation}\label{eq:difference}
 F(\mathbf u,\mathbf v)=P(\mathbf u)-P(\mathbf v),
 \qquad F_{d}(\mathbf u,\mathbf v)=H(\mathbf u)-H(\mathbf v).
\end{equation}
Since the variables are disjoint,
\begin{equation}\label{eq:rankdoubling}
 \Crit(F_{d})=\Crit(H)\times\Crit(H),
 \qquad \B(F_{d})=2\B(H)=2\B(P)>({d}-1)2^{d}.
\end{equation}
Thus Proposition~\ref{input:Birch} applies. Write
\begin{equation}\label{eq:localP}
 C_P(q,a)=\sum_{\mathbf n\bmod q}e(aP(\mathbf n)/q),
 \qquad V_H(\gamma)=\int_{[0,1]^k}e(\gamma H(\mathbf u))\dd\mathbf u.
\end{equation}
Then by the definition of the difference polynomials in \eqref{eq:difference}, we have
\begin{equation}\label{eq:positivefactors}
 C_F(q,a)=|C_P(q,a)|^2,
 \qquad V_{F_{d}}(\gamma)=|V_H(\gamma)|^2.
\end{equation}
In particular,
\begin{equation}\label{eq:positiveconstants}
 \SingSeries_P=\sum_{q\ge1}\sum_{\substack{0\le a<q\\(a,q)=1}}
 q^{-2k}|C_P(q,a)|^2<\infty,
 \qquad \JJ_H=\int_{\R}|V_H(\gamma)|^2\dd\gamma<\infty.
\end{equation}

\begin{proposition}
\label{prop:Birchenergy}
Under \eqref{eq:rankcondition}, hypothesis
\eqref{eq:energy} holds.
\end{proposition}
\begin{proof}
Recall from \eqref{eq:SPintro} that
\[
S_{P,N}(\alpha)=\sum_{\mathbf n\in[N]^k}e(\alpha P(\mathbf n)).
\]
Expanding the following square, we have
\[
 \begin{aligned}
 \int_{\T}|S_{P,N}(\alpha)|^2\dd\alpha
 &=\sum_{\mathbf u,\mathbf v\in[N]^k}
   \int_{\T}e\bigl(\alpha(P(\mathbf u)-P(\mathbf v))\bigr)
        \dd\alpha\\
 &=\#\{(\mathbf u,\mathbf v)\in[N]^{2k}:
                      P(\mathbf u)=P(\mathbf v)\}.
 \end{aligned}
\]
Applying Proposition~\ref{input:Birch} to $F$ defined by \eqref{eq:difference}, we get that
\begin{equation}\label{eq:totalenergy}
 N^{{d}-2k}\int_{\T}|S_{P,N}(\alpha)|^2\dd\alpha
 =\SingSeries_P\JJ_H+O_P(N^{-\delta}).
\end{equation}

Next fix $Q\ge1$, integers $1\le q\le Q$ and $a$, and consider $|\gamma|\le Q$. Splitting
the sum into residue classes modulo $q$ gives
\begin{equation} \label{S_PN_decom}
 N^{-k}S_{P,N}(a/q+\gamma N^{-{d}}) =\sum_{\mathbf r\bmod q}e(aP(\mathbf r)/q)
    N^{-k}\sum_{\substack{\mathbf n\in[N]^k\\
                         \mathbf n\equiv\mathbf r\pmod q}}
      e\bigl(\gamma N^{-{d}}P(\mathbf n)\bigr).
\end{equation}

Write $P=H+P_{{d}-1}+\cdots+P_0$ as a sum of homogeneous parts.
Uniformly for $\mathbf u\in[0,1]^k$,
\[
 N^{-{d}}P(N\mathbf u)
 =H(\mathbf u)+\sum_{j=0}^{{d}-1}N^{j-{d}}P_j(\mathbf u)
 =H(\mathbf u)+O_P(N^{-1}).
\]
Since $|e(v)-e(w)|\le2\pi|v-w|$ for real $v,w$, we have
\[
e(\gamma N^{-{d}}P(\mathbf n)) = e(\gamma H(\mathbf n/N))+O_{P,Q}(N^{-1})
\]
uniformly for $|\gamma|\le Q$.

Write
\[
 f_\gamma(\mathbf u)=e(\gamma H(\mathbf u)),
 \qquad
 V_H(\gamma)=\int_{[0,1]^k}f_\gamma(\mathbf u)\,d\mathbf u,
 \qquad h=\frac qN.
\]
Then $|f_\gamma(\mathbf u)|=1$ and $0<h\le1$ for $N\ge Q$.

For any $\mathbf u,\mathbf v\in[0,1]^k$, we use $\|\mathbf u-\mathbf v\|_2$ to denote the Euclidean norm.
By
\[
 \frac{\partial f_\gamma}{\partial u_j}(\mathbf u)
 =
 2\pi i\gamma
 \frac{\partial H}{\partial u_j}(\mathbf u)
 e(\gamma H(\mathbf u)),
\]
we have
\[
 \sup_{\mathbf u\in[0,1]^k}
 \|\nabla f_\gamma(\mathbf u)\|_2
 \le
 2\pi Q
 \sup_{\mathbf u\in[0,1]^k}\|\nabla H(\mathbf u)\|_2
 =:L_{H,Q}
\]
 for all $|\gamma|\le Q$.
This gives the following uniform Lipschitz estimate
\[
 |f_\gamma(\mathbf u)-f_\gamma(\mathbf v)|
 \le L_{H,Q}\|\mathbf u-\mathbf v\|_2
 \qquad(\mathbf u,\mathbf v\in[0,1]^k).
\]

Choose the representatives $r_j\in\{1,\ldots,q\}$, so that $q$ represents the zero residue class. For each coordinate, the integers in $[N]$ satisfying $n_j\equiv r_j\pmod q$
are
\[
 r_j,\ r_j+q,\ \ldots,\ r_j+(m_j-1)q,
 \qquad
 m_j=\left\lfloor\frac{N-r_j}{q}\right\rfloor+1.
\]
Let
\[
 b_j=r_j+(m_j-1)q
\]
be the largest such integer. Then
\[
 N-q<b_j\le N.
\]
Set
\[
 \mathcal G=
 \{\mathbf n\in[N]^k:\mathbf n\equiv\mathbf r\pmod q\}.
\]
For each $\mathbf n\in\mathcal G$, define the truncated cell
\[
 C_{\mathbf n}
 =
 \prod_{j=1}^k
 \left(
 \max\left\{0,\frac{n_j-q}{N}\right\},
 \frac{n_j}{N}
 \right].
\]
These cells are pairwise disjoint and satisfy
\[
 \bigcup_{\mathbf n\in\mathcal G}C_{\mathbf n}
 =
 D_N:=\prod_{j=1}^k\left(0,\frac{b_j}{N}\right].
\]
Every side length is at most $h$, and Cauchy-Schwarz gives
\[
 \left\|\frac{\mathbf n}{N}-\mathbf u\right\|_2
 \le\sqrt{k}\,h
 \qquad(\mathbf u\in C_{\mathbf n}).
\]
The Lipschitz estimate yields
\[
 \left|
 |C_{\mathbf n}|f_\gamma(\mathbf n/N)
 -\int_{C_{\mathbf n}}f_\gamma(\mathbf u)\,d\mathbf u
 \right|\le
 \int_{C_{\mathbf n}}
 |f_\gamma(\mathbf n/N)-f_\gamma(\mathbf u)|\,d\mathbf u
 \le L_{H,Q}\sqrt{k}\,h\,|C_{\mathbf n}|.
\]
Summing over $\mathcal G$ and using $|D_N|\le1$, we obtain
\[
 \left|
 \sum_{\mathbf n\in\mathcal G}
 |C_{\mathbf n}|f_\gamma(\mathbf n/N)
 -\int_{D_N}f_\gamma(\mathbf u)\,d\mathbf u
 \right|
 \le L_{H,Q}\sqrt{k}\,h.
\]

It remains to control two boundary errors. First, a cell can have
volume smaller than $h^k$ only if $n_j=r_j$ for at least one
coordinate $j$. Since $m_i\le h^{-1}+1$, the number of such cells
is at most
\[
 \sum_{j=1}^k\prod_{i\ne j}m_i
 \le k(h^{-1}+1)^{k-1}.
\]
Each cell has volume deficit at most $h^k$. Thus
\[
 \begin{aligned}
 &\left|
 h^k\sum_{\mathbf n\in\mathcal G}f_\gamma(\mathbf n/N)
 -\sum_{\mathbf n\in\mathcal G}
 |C_{\mathbf n}|f_\gamma(\mathbf n/N)
 \right|\\
 &\qquad\le
 \sum_{\mathbf n\in\mathcal G}(h^k-|C_{\mathbf n}|)
 \le k(h^{-1}+1)^{k-1}h^k
 =kh(1+h)^{k-1}
 \le k2^{k-1}h.
 \end{aligned}
\]
Second, $0\le1-b_j/N<h$ for every $j$, so
\[
 |[0,1]^k\setminus D_N|
 =
 1-\prod_{j=1}^k\frac{b_j}{N}
 \le
 \sum_{j=1}^k\left(1-\frac{b_j}{N}\right)
 \le kh.
\]
Since $|f_\gamma|=1$, it follows that
\[
 \left|
 V_H(\gamma)
 -\int_{D_N}f_\gamma(\mathbf u)\,d\mathbf u
 \right|
 \le kh.
\]

Combining these estimates gives
\[
 \left|
 h^k\sum_{\mathbf n\in\mathcal G}f_\gamma(\mathbf n/N)
 -V_H(\gamma)
 \right|
 \le
 \bigl(\sqrt{k}L_{H,Q}+k2^{k-1}+k\bigr)h.
\]
Therefore
\[
 \left(\frac qN\right)^k
 \sum_{\substack{\mathbf n\in[N]^k\\
                  \mathbf n\equiv\mathbf r\pmod q}}
 e\bigl(\gamma H(\mathbf n/N)\bigr)
 =
 V_H(\gamma)+O_{H,Q,k}\left(\frac qN\right).
\]
Multiplying by $q^{-k}$, we conclude that
\[
 N^{-k}
 \sum_{\substack{\mathbf n\in[N]^k\\
                  \mathbf n\equiv\mathbf r\pmod q}}
 e\bigl(\gamma H(\mathbf n/N)\bigr)
 =
 q^{-k}V_H(\gamma)
 +O_{H,Q,k}\left(\frac{q^{1-k}}N\right).
\]
In particular, for fixed $q$, the error is
$O_{P,Q,q}(N^{-1})$. All estimates are uniform for $|\gamma|\le Q$
and for all residue classes $\mathbf r\bmod q$.

By \eqref{S_PN_decom}, summing the
last approximation with their rational phases,
\begin{align}
N^{-k}S_{P,N}(a/q+\gamma N^{-{d}})
&=
q^{-k}V_H(\gamma)
\sum_{\mathbf r\bmod q}e\bigl(aP(\mathbf r)/q\bigr)
+O_{P,Q,q}(N^{-1}) \nonumber\\
&=
q^{-k}C_P(q,a)V_H(\gamma)
+O_{P,Q,q}(N^{-1}). \label{eq:fixedmajorapprox}
\end{align}
holds uniformly for $|\gamma|\le Q$.

Fix now $Q$. Suppose that $2QN^{-{d}}<Q^{-2}$, then the arcs in $\M_N(Q)$ are disjoint. Changing
variables by $\alpha=a/q+\gamma N^{-{d}}$, we have
\[
 N^{{d}-2k}\int_{\M_N(Q)}|S_{P,N}(\alpha)|^2\dd\alpha=\sum_{q\le Q}\sum_{\substack{0\le a<q\\(a,q)=1}}
   \int_{-Q}^Q
   \left|N^{-k}S_{P,N}(a/q+\gamma N^{-{d}})\right|^2
      \dd\gamma.
\]
Since  \eqref{eq:fixedmajorapprox} holds uniformly for $|\gamma|\le Q$, taking $N\to\infty$ above gives
\begin{equation}\label{eq:truncatedenergy}
 \lim_{N\to\infty}N^{{d}-2k}
 \int_{\M_N(Q)}|S_{P,N}(\alpha)|^2\dd\alpha
 =\SingSeries_P(Q)\JJ_H(Q),
\end{equation}
where
\[
 \SingSeries_P(Q)=\sum_{q\le Q}\sum_{\substack{0\le a<q\\(a,q)=1}}
 q^{-2k}|C_P(q,a)|^2,
 \qquad \JJ_H(Q)=\int_{-Q}^Q|V_H(\gamma)|^2\dd\gamma.
\]
It follows by \eqref{eq:totalenergy} that
\begin{equation}\label{eq:taildifference}
 \lim_{N\to\infty}N^{{d}-2k}\int_{\m_N(Q)}|S_{P,N}(\alpha)|^2\dd\alpha
 =\SingSeries_P\JJ_H-\SingSeries_P(Q)\JJ_H(Q).
\end{equation}

Taking $Q\to\infty$ in \eqref{eq:taildifference} proves \eqref{eq:energy}.
\end{proof}

\begin{proof}[Proof of Theorem~\ref{thm:Birch}]
Combine Corollary~\ref{lem:arithmetic},
Proposition~\ref{prop:transfer}, and Proposition~\ref{prop:Birchenergy}.
\end{proof}

\section{Proof of Theorem~\ref{thm:diagonal} }\label{sec:Weyl}

For $d\ge3$ put
\begin{equation}\label{eq:Weyl}
 W_{d,N}(\alpha)=\sum_{n=1}^N e(\alpha n^d).
\end{equation}
Recall that
for $Q\ge1$, we have
\[
 \mathfrak M_N(Q)
 =
 \bigcup_{\substack{1\le q\le Q\\0\le a<q, (a,q)=1}}
 \left\{\alpha\in\mathbb T:
 \left\|\alpha-\frac aq\right\|_{\mathbb T}
 \le QN^{-d}\right\},
 \qquad
 \mathfrak m_N(Q)=\mathbb T\setminus\mathfrak M_N(Q).
\]

\begin{proposition}\label{input:endpoint}
Let $d\ge3$ and $p=2^d$. Then
\begin{equation}\label{eq:endpoint}
 \lim_{Q\to\infty}\limsup_{N\to\infty}
 N^{d-p}\int_{\m_N(Q)}|W_{d,N}(\alpha)|^p\dd\alpha=0.
\end{equation}
\end{proposition}

\begin{proof}

We prove \eqref{eq:endpoint} separately for $d=3$ and $d\ge4$. 

\medskip

\noindent\textit{The  case $d=3$.}
Let
\[
 \widehat{\mathfrak M}_N
 =
 \bigcup_{\substack{1\le q\le N^{3/4}\\0\le a<q\\(a,q)=1}}
 \widehat{\mathfrak M}_N(q,a),
\]
where
\[
 \widehat{\mathfrak M}_N(q,a)
 =
 \left\{\frac aq+\beta\pmod1:
 |\beta|\le q^{-1}N^{-9/4}\right\},
 \qquad
 \widehat{\mathfrak m}_N=\mathbb T\setminus\widehat{\mathfrak M}_N.
\]
In this notation, the mean value estimates  in
\cite[(2.6), (2.8)]{BW} give
\begin{equation}\label{eq:cubic-broad-minor}
 \int_{\widehat{\mathfrak m}_N}|W_{3,N}(\alpha)|^8\,d\alpha
 \ll N^5(\log N)^{-2}
\end{equation}
and, whenever $a/q+\beta\in\widehat{\mathfrak M}_N(q,a)$,
\begin{equation}\label{eq:cubic-broad-pointwise}
 |W_{3,N}(a/q+\beta)|
 \ll Nq^{-1/3}(1+N^3|\beta|)^{-1/3}.
\end{equation}

Decompose the integral as
\[
 \int_{\mathfrak m_N(Q)}|W_{3,N}(\alpha)|^8\,d\alpha
 =
 \int_{\mathfrak m_N(Q)\cap\widehat{\mathfrak m}_N}
 |W_{3,N}(\alpha)|^8\,d\alpha
 +
 \int_{\mathfrak m_N(Q)\cap\widehat{\mathfrak M}_N}
 |W_{3,N}(\alpha)|^8\,d\alpha.
\]
The first term is bounded by \eqref{eq:cubic-broad-minor}.
For the second term, observe that if
\[
 \alpha=\frac aq+\beta
 \in\mathfrak m_N(Q)\cap\widehat{\mathfrak M}_N(q,a),
\]
then either $q>Q$, or $q\le Q$ and $|\beta|>QN^{-3}$. Raising \eqref{eq:cubic-broad-pointwise} to the eighth power
and making the substitution $u=N^3\beta$, we obtain
\[
 N^{-5}|W_{3,N}(a/q+\beta)|^8\,d\beta
 \ll q^{-8/3}(1+|u|)^{-8/3}\,du.
\]

The contribution from denominators $q>Q$ is at most
\[
 \begin{aligned}
 &N^{-5}
 \sum_{Q<q\le N^{3/4}}
 \sum_{\substack{0\le a<q\\(a,q)=1}}
 \int_{\widehat{\mathfrak M}_N(q,a)}
 |W_{3,N}(\alpha)|^8\,d\alpha\\
 &\qquad\ll
 \sum_{q>Q}\varphi(q)q^{-8/3}
 \int_{\mathbb R}(1+|u|)^{-8/3}\,du\\
 &\qquad\ll
 \sum_{q>Q}q^{-5/3}
 \ll Q^{-2/3}.
 \end{aligned}
\]

For $q\le Q$, the contribution from $|\beta|>QN^{-3}$
is at most
\[
 \begin{aligned}
 &N^{-5}
 \sum_{q\le Q}
 \sum_{\substack{0\le a<q\\(a,q)=1}}
 \int_{\substack{QN^{-3}<|\beta|\le q^{-1}N^{-9/4}}}
 |W_{3,N}(a/q+\beta)|^8\,d\beta\\
 &\qquad\ll
 \sum_{q\le Q}q^{-5/3}
 \int_{|u|>Q}(1+|u|)^{-8/3}\,du
 \ll Q^{-5/3}.
 \end{aligned}
\]

Combining these estimates gives
\begin{equation}\label{eq:cubicendpoint}
 N^{-5}\int_{\mathfrak m_N(Q)}
 |W_{3,N}(\alpha)|^8\,d\alpha
 \ll (\log N)^{-2}+Q^{-2/3},
\end{equation}
for every fixed $Q\ge1$ and all sufficiently large $N$.
Letting $N\to\infty$ and then $Q\to\infty$ in \eqref{eq:cubicendpoint} proves  \eqref{eq:endpoint}  for $d=3$.

\medskip
\noindent\textit{The case $d\ge4$.}
Notice that
\[
\int_{\mathbb T}|W_{d,N}(\alpha)|^p\,d\alpha
=
\#\left\{
(\mathbf x,\mathbf y)\in[N]^s\times[N]^s:
\sum_{j=1}^s x_j^d=\sum_{j=1}^s y_j^d
\right\},
\qquad s=p/2=2^{d-1}.
\]

Write
\[
 C_d(q,a)=\sum_{r\bmod q}e(ar^d/q),
 \qquad
 V_d(\gamma)=\int_0^1e(\gamma t^d)\,dt.
\]
For every $\varepsilon>0$, Vaughan's  mean-value asymptotic
\cite[Theorem 2]{Vaughan} states that
\begin{equation}\label{eq:Vaughantotal}
 \int_{\mathbb T}|W_{d,N}(\alpha)|^p\,d\alpha
 =
 \mathfrak S_d\mathfrak J_d N^{p-d}
 +O_{d,\varepsilon}\bigl(N^{p-d}(\log N)^{\varepsilon-2}\bigr),
\end{equation}
where
\[
 \mathfrak S_d
 =
 \sum_{q=1}^{\infty}
 \sum_{\substack{0\le a<q\\(a,q)=1}}
 |q^{-1}C_d(q,a)|^p,
 \qquad
 \mathfrak J_d
 =
 \int_{\mathbb R}|V_d(\gamma)|^p\,d\gamma.
\]

Similar to the proof of \eqref{eq:taildifference} in Proposition~\ref{prop:Birchenergy},  for each fixed $Q$, by \eqref{eq:Vaughantotal} we have
\[
 \lim_{N\to\infty}
 N^{d-p}\int_{\mathfrak m_N(Q)}
 |W_{d,N}(\alpha)|^p\,d\alpha
 =
 \mathfrak S_d\mathfrak J_d
 -\mathfrak S_d(Q)\mathfrak J_d(Q),
\]
where
\[
 \mathfrak S_d(Q)
 =
 \sum_{q\le Q}
 \sum_{\substack{0\le a<q\\(a,q)=1}}
 |q^{-1}C_d(q,a)|^p,
 \qquad
 \mathfrak J_d(Q)
 =
 \int_{-Q}^{Q}|V_d(\gamma)|^p\,d\gamma.
\]
Then \eqref{eq:endpoint}  follows by taking $Q\to\infty$.
\end{proof}

For $k\ge2^{d-1}$, using the trivial bound $|W_{d,N}|\le N$, we have
$|W_{d,N}|^{2k}\le N^{2k-2^d}|W_{d,N}|^{2^d}$. By Proposition~\ref{input:endpoint}, we obtain the following corollary.

\begin{corollary}\label{cor:classicaldiagonal}
For $d\ge3$ and $k\ge2^{d-1}$,
\begin{equation}\label{eq:diagonalenergy}
 \lim_{Q\to\infty}\limsup_{N\to\infty}
 N^{d-2k}\int_{\m_N(Q)}|W_{d,N}(\alpha)|^{2k}\dd\alpha=0.
\end{equation}
\end{corollary}

Now, we optimize the interpolation for every $d\ge4$. Define the larger
major arcs
\begin{equation}\label{eq:largearcs}
 \begin{gathered}
 \M_N^*(q,a)=\left\{\alpha\in\T:
 \left\|\alpha-\frac aq\right\|\le\frac{N^{1-d}}{2dq}\right\},\\
 \M_N^*=\bigcup_{\substack{1\le q\le N/(2d)\\0\le a<q,\ (a,q)=1}}
 \M_N^*(q,a),\qquad \m_N^*=\T\setminus\M_N^*.
 \end{gathered}
\end{equation}
The major arcs are pairwise disjoint for sufficiently
large $N$. For fixed $Q$ and $N\ge2dQ^2$, one has
$\M_N(Q)\subset\M_N^*$; Proposition~\ref{prop:transfer} still uses the
fixed arcs $\M_N(Q)$, with $N\to\infty$ before $Q\to\infty$.

\begin{proposition}\label{input:modern}
For $d\ge4$ and every $\eta>0$,
\begin{align}
 & \int_{\m_N^*}|W_{d,N}(\alpha)|^{d(d+1)}\dd\alpha
\ll_{d,\eta}N^{d(d+1)-d-1+\eta},\label{eq:VMVTminor}\\
& \int_{\T}|W_{d,N}(\alpha)|^{a_j}\dd\alpha
 \ll_{d,\eta}N^{a_j-j+\eta}
 \quad(1\le j\le d-1),\label{eq:lowmoments}
\end{align}
where $a_j=\min\{2^j,j(j+1)\}$. 

On the arc of
\eqref{eq:largearcs} centered at $a/q$, put
\[
 C_d(q,a)=\sum_{r\bmod q}e(ar^d/q),\qquad
 v_{d,N}(\beta)=\int_0^N e(\beta t^d)\dd t.
\]
The standard major-arc approximation and estimates are
\begin{align}
 W_{d,N}(a/q+\beta)
 &=q^{-1}C_d(q,a)v_{d,N}(\beta)
 +O_{d,\eta}\!\left(q^{1/2+\eta}(1+N^d|\beta|)^{1/2}\right),
 \label{eq:majorapproxWeyl}\\
 |q^{-1}C_d(q,a)|&\ll_{d,\eta}q^{-1/d+\eta},\qquad
 |v_{d,N}(\beta)|\ll_d N(1+N^d|\beta|)^{-1/d}.
 \label{eq:majorWeyl}
\end{align}
\end{proposition}

\begin{proof}

Estimate \eqref{eq:VMVTminor} follows directly from
\cite[Theorem 14.4(i)]{WooleyNested}.
For \eqref{eq:lowmoments}, Hua's lemma  \cite{Hua}  and Bourgain's estimate \cite[Theorem 10]{Bourgain}, respectively, give
\[
 \int_{\T}|W_{d,N}|^{2^j}\dd\alpha
 \ll_{d,\eta}N^{2^j-j+\eta},\qquad
 \int_{\T}|W_{d,N}|^{j(j+1)}\dd\alpha
 \ll_{d,\eta}N^{j^2+\eta}
 \quad(1\le j\le d).
\]
So \eqref{eq:lowmoments} holds for $a_j=\min\{2^j,j(j+1)\}$. 

Equations \eqref{eq:majorapproxWeyl} and \eqref{eq:majorWeyl} follow by
\cite[Theorems 4.1 and 4.2]{HL}.
\end{proof}

\begin{lemma}\label{lem:interpolation}
Let $d\ge4$, and define
\[
 D_d=d(d+1),\qquad
 \theta_{d,j}=D_d-\frac{D_d-a_j}{d+1-j},\qquad
 \Theta_d=\min_{1\le j\le d-1}\theta_{d,j}.
\]
If
\begin{equation}\label{eq:pcondition}
 p>\Theta_d,
\end{equation}
then there is $\delta=\delta(d,p)>0$ such that
\begin{equation}\label{eq:pminor}
 \int_{\m_N^*}|W_{d,N}(\alpha)|^p\dd\alpha
 \ll_{d,p}N^{p-d-\delta}.
\end{equation}
Moreover, $\Theta_d>D_d/2>2d$.
\end{lemma}
\begin{proof}
For $1\le j\le d-1$ we have $0<a_j<D_d$ and $h_j:=d+1-j\ge2$.
Consequently
\[
 a_j<\theta_{d,j}<D_d,\qquad
 \theta_{d,j}-\frac{D_d}{2}
 =\left(\frac12-\frac1{h_j}\right)D_d+\frac{a_j}{h_j}>0.
\]
Taking the minimum proves $\Theta_d>D_d/2>2d$ for $d\ge4$.

Choose $j$ attaining $\Theta_d$. If $\Theta_d<p<D_d$, put
\[
 \lambda=\frac{D_d-p}{D_d-a_j}\in(0,1),\qquad
 p=\lambda a_j+(1-\lambda)D_d.
\]
H\"older's inequality and
\eqref{eq:lowmoments}, \eqref{eq:VMVTminor} give
\begin{align*}
 \int_{\m_N^*}|W_{d,N}|^p
 &\le
 \left(\int_{\T}|W_{d,N}|^{a_j}\right)^\lambda
 \left(\int_{\m_N^*}|W_{d,N}|^{D_d}\right)^{1-\lambda}\\
 &\ll_{d,p,\eta}
 N^{\lambda(a_j-j)+(1-\lambda)(D_d-d-1)+\eta}
 =N^{p-d-\Delta+\eta},
\end{align*}
where
\[
 \Delta=1-\frac{(D_d-p)(d+1-j)}{D_d-a_j}>0
 \quad\Longleftrightarrow\quad
 p>\theta_{d,j}=\Theta_d.
\]
Take $\eta=\Delta/2$ and $\delta=\Delta/2$. If $p\ge D_d$,
the trivial bound $|W_{d,N}|\le N$ instead gives
\[
 \int_{\m_N^*}|W_{d,N}|^p
 \le N^{p-D_d}\int_{\m_N^*}|W_{d,N}|^{D_d}
 \ll_{d,p,\eta}N^{p-d-1+\eta}.
\]
Taking $\eta=1/2$ proves the claim, including $p=D_d$.
\end{proof}

\begin{lemma}\label{lem:pruning}
Suppose $p>2d$ and \eqref{eq:pminor} holds. Then
\begin{equation}\label{eq:pfixed}
 \lim_{Q\to\infty}\limsup_{N\to\infty}
 N^{d-p}\int_{\m_N(Q)}|W_{d,N}(\alpha)|^p\dd\alpha=0.
\end{equation}
\end{lemma}
\begin{proof}

For sufficiently large $N$, we have $\mathfrak M_N(Q)\subset \mathfrak M_N^*$, and
\(
 \mathfrak m_N(Q)
 =
 \mathfrak m_N^*
 \,\cup\,
 \bigl(\mathfrak M_N^*\setminus\mathfrak M_N(Q)\bigr).
\)
By Lemma~\ref{lem:interpolation}, we have
\[
 N^{d-p}\int_{\mathfrak m_N^*}
 |W_{d,N}(\alpha)|^p\,d\alpha
 \ll_{d,p} N^{-\delta}.
\]
So it remains to estimate the contribution from
$\mathfrak M_N^*\setminus\mathfrak M_N(Q)$.

Since each
$\alpha\in\mathfrak M_N^*$ belongs to a unique arc
$\mathfrak M_N^*(q,a)$, write
\[
 \alpha=\frac aq+\beta\pmod1,
 \qquad
 |\beta|\le\frac{N^{1-d}}{2dq},
\]
and define
\[
 A_N(\alpha)
 =A_{q,a}(\beta):=q^{-1}C_d(q,a)v_{d,N}(\beta),
 \qquad
 E_N(\alpha)
 :=W_{d,N}(\alpha)-A_N(\alpha).
\]
Then
\[
 W_{d,N}(\alpha)=A_N(\alpha)+E_N(\alpha)
 \qquad (\alpha\in\mathfrak M_N^*).
\]

By $q\leq N/(2d)$ and $|\beta| \leq N^{1-d}/(2qd)$,
\[
 q(1+N^d|\beta|)
 =
 q+N^dq|\beta|
 \le\frac N{2d}+\frac N{2d}
 =\frac Nd.
\]
By \eqref{eq:majorapproxWeyl}, we have
\[
 |E_N(\alpha)|
 \ll_{d,\eta}
 q^{1/2+\eta}(1+N^d|\beta|)^{1/2}
 \ll_{d,\eta}N^{1/2+\eta}
\]
uniformly on $\mathfrak M_N^*$.
It follows that 
\[
 \begin{aligned}
 N^{d-p}\int_{\mathfrak M_N^*}|E_N(\alpha)|^p\,d\alpha
 &\ll_{d,p,\eta}
 N^{d-p}N^{p/2+p\eta}=N^{d-p/2+p\eta}.
 \end{aligned}
\]

For $p\ge1$, by the convexity inequality
\(
 (x+y)^p\le2^{p-1}(x^p+y^p)
\)
 for any $x,y\ge0$, we have
\[
 |W_{d,N}(\alpha)|^p
 \le\bigl(|A_N(\alpha)|+|E_N(\alpha)|\bigr)^p
 \le2^{p-1}\bigl(|A_N(\alpha)|^p+|E_N(\alpha)|^p\bigr).
\]
So it remains to estimate the main term.

By \eqref{eq:majorWeyl},
\[
 |A_{q,a}(\beta)|^p
 \ll_{d,p,\eta}
 N^p q^{-p/d+p\eta}
 (1+N^d|\beta|)^{-p/d}.
\]
Making the substitution $u=N^d\beta$, we find that, for
any measurable set $I$ contained in the corresponding
$\beta$-interval 
\(
\left[-\frac{N^{1-d}}{2dq},
       \frac{N^{1-d}}{2dq}\right],
\)
\[
 \begin{aligned}
 N^{d-p}\int_I|A_{q,a}(\beta)|^p\,d\beta
 &\ll_{d,p,\eta}
 q^{-p/d+p\eta}
 \int_{N^d I}(1+|u|)^{-p/d}\,du.
 \end{aligned}
\]

For the contribution from $q>Q$, we have
\[
 \begin{aligned}
 &N^{d-p}
 \sum_{Q<q\le N/(2d)}
 \sum_{\substack{0\le a<q\\(a,q)=1}}
 \int_{|\beta|\le N^{1-d}/(2dq)}
 |A_{q,a}(\beta)|^p\,d\beta\\
 &\qquad\ll_{d,p,\eta}
 \sum_{q>Q}q^{1-p/d+p\eta}
 \int_{\mathbb R}(1+|u|)^{-p/d}\,du\\
 &\qquad\ll_{d,p,\eta}
 Q^{2-p/d+p\eta}.
 \end{aligned}
\]

Next, suppose that $q\le Q$ and
\[
 \alpha=\frac aq+\beta
 \in\mathfrak M_N^*(q,a)\setminus\mathfrak M_N(Q).
\]
Then $|\beta|>QN^{-d}$, $|u|>Q$, and the contribution from these
arcs is at most
\[
 \begin{aligned}
 &N^{d-p}
 \sum_{q\le Q}
 \sum_{\substack{0\le a<q\\(a,q)=1}}
 \int_{\substack{|\beta|\le N^{1-d}/(2dq)\\
                 a/q+\beta\notin\mathfrak M_N(Q)}}
 |A_{q,a}(\beta)|^p\,d\beta\\
 &\qquad\ll_{d,p,\eta}
 \sum_{q\le Q}q^{1-p/d+p\eta}
 \int_{|u|>Q}(1+|u|)^{-p/d}\,du\\
 &\qquad\ll_{d,p,\eta}Q^{1-p/d}.
 \end{aligned}
\]

Combining the preceding estimates, we obtain, for fixed
$Q$ and sufficiently large $N$,
\[
 N^{d-p}\int_{\mathfrak m_N(Q)}
 |W_{d,N}(\alpha)|^p\,d\alpha \ll_{d,p,\eta}
 N^{-\delta}
 +N^{d-p/2+p\eta}
 +Q^{2-p/d+p\eta}
 +Q^{1-p/d},
\]
where $\eta>0$ is sufficiently small.
Then \eqref{eq:pfixed} follows by letting  $N\to\infty$ first  and then 
$Q\to\infty$.
\end{proof}

\begin{proposition}\label{prop:modernenergy}
For $d\ge4$ and $k\ge K(d)=\lfloor\Theta_d/2\rfloor+1$, the
 property \eqref{eq:diagonalenergy} holds. That is,
\begin{equation*}
 \lim_{Q\to\infty}\limsup_{N\to\infty}
 N^{d-2k}\int_{\m_N(Q)}|W_{d,N}(\alpha)|^{2k}\dd\alpha=0.
\end{equation*}
\end{proposition}
\begin{proof}
The integer $2K(d)$ is the smallest even integer strictly greater
than $\Theta_d$. Hence $p=2k>\Theta_d>2d$, so
Lemmas~\ref{lem:interpolation} and \ref{lem:pruning} apply.
\end{proof}

\begin{proposition}
\label{prop:thresholdcomparison}
For $d\ge5$,
\[
 K(d)\le\left\lceil\frac{d^2-2}{2}\right\rceil,
\]
and for every $d\ge4$,
\[
 K(d)\le\frac{d(d-1)}2+\left\lfloor\sqrt{2d+2}\right\rfloor.
\]
\end{proposition}
\begin{proof}
The choice $j=3$, with $a_3=8$, gives
\[
 \Theta_d\le d(d+1)-\frac{d(d+1)-8}{d-2}
 =d^2-3+\frac2{d-2}<d^2-2\qquad(d\ge5).
\]
This proves the first assertion. When $d=4$, $K(4)=8$.

For the second assertion write $h=d+1-j$. Since
$a_j\le j(j+1)$, expansion gives
\[
 \Theta_d\le D_d-\frac{D_d-j(j+1)}h
 =d^2-d-3+h+\frac{2d+2}{h}.
\]
Choose $h=\lfloor\sqrt{2d+2}\rfloor$, which lies in $[2,d]$.
Then $2d+2<(h+1)^2$ implies
$2d+2\le h^2+2h$. Hence
\[
 \Theta_d\le d^2-d+2h-1<d(d-1)+2h.
\]
This proves the second assertion.
\end{proof}

\begin{proof}[Proof of Theorem~\ref{thm:diagonal}]

For $\boldsymbol\varepsilon=(\varepsilon_1,\ldots,\varepsilon_k)
\in\{1,-1\}^k$, let $k_+$ and $k_-$ count the plus and minus signs in $\boldsymbol\varepsilon$, so
$k_++k_-=k$. For the polynomial
$P_{\boldsymbol\varepsilon}(\mathbf n)=\sum_{j=1}^k\varepsilon_jn_j^{d}$, we have
\[
 S_{P_{\boldsymbol\varepsilon},N}(\alpha)
 =\sum_{\mathbf n\in[N]^k}e(\alpha P_{\boldsymbol\varepsilon}(\mathbf n))=\prod_{j=1}^k\left(\sum_{n_j=1}^N
                         e(\varepsilon_j\alpha n_j^{d})\right).
\]

By
$W_{{d},N}(-\alpha)=\overline{W_{{d},N}(\alpha)}$, we obtain that
\begin{equation}\label{eq:signfactor}
 S_{P_{\boldsymbol\varepsilon},N}(\alpha)
 =W_{{d},N}(\alpha)^{k_+}\overline{W_{{d},N}(\alpha)}^{\,k_-},
 \qquad |S_{P_{\boldsymbol\varepsilon},N}|^2=|W_{{d},N}|^{2k}.
\end{equation}
For ${d}=3$, Corollary~\ref{cor:classicaldiagonal} proves
\eqref{eq:energy}; for ${d}\ge4$,
Proposition~\ref{prop:modernenergy} proves
\eqref{eq:energy}. Proposition~\ref{prop:transfer}
then gives  \eqref{eq:diagonalmain}.

 If ${d}=2$ and $k\ge3$, the diagonal
quadratic form has rank $k$, so Theorem~\ref{thm:Birch} applies. For ${d}=k=2$, by \cite[Theorem D]{DLMS}, \eqref{eq:diagonalmain} holds for $P_{\boldsymbol\varepsilon}(\mathbf n) = n_1^2+n_2^2$. By the following Proposition~\ref{prop:diffsquares}, it also holds for $P_{\boldsymbol\varepsilon}(\mathbf n) = n_1^2-n_2^2$. This completes the proof of Theorem~\ref{thm:diagonal}.
\end{proof}

\begin{proposition}\label{prop:diffsquares}
For every uniquely ergodic system $(X,\nu,T)$, every $x\in X$, and $f\in C(X)$,
\begin{equation}\label{eq_diffsquares}
 \lim_{N\to\infty}N^{-2}\sum_{m,n\le N}
 f(T^{\Omega(n^2-m^2)}x)=\av f.
\end{equation}
\end{proposition}

\begin{proof}

Let $f\in C(X)$. By \cite[Theorem 1.1]{Richter},  we get that \eqref{eq:BRintro}  holds uniformly for $x\in X$. That is,
\begin{equation}\label{eq:uniformBR}
 \lim_{M\to\infty}\sup_{y\in X}
 \Big|\frac1M\sum_{n\le M}f(T^{\Omega(n)}y)-\av f\Big|=0.
\end{equation}
Then \eqref{eq:uniformBR} holds for all even $n$ and hence for all odd $n$.  Set $h=f-\av f$, then  $\av h=0$. It follows that for $r=0,1$,
\begin{equation}\label{eq:parity}
 \sup_{y\in X}\max_{0\le L\le2N}
 \Big|\sum_{\substack{1\le v\le L\\v\equiv r\ (2)}}
 h(T^{\Omega(v)}y)\Big|=o(N).
\end{equation}

Now, we prove that
\[
 \lim_{N\to\infty}N^{-2}\sum_{m,n\le N}
 f(T^{\Omega(n^2-m^2)}x)=\av f.
\]

The diagonal $m=n$ contributes $O_f(N)$ to the sum. On $n>m$, put $u=n-m$, then
\begin{align*}
 \sum_{\substack{m,n\le N\\m\ne n}}
 h(T^{\Omega(|n^2-m^2|)}x)
 &=2\sum_{u=1}^{N-1}\sum_{m=1}^{N-u}
 h\bigl(T^{\Omega(u+2m)}(T^{\Omega(u)}x)\bigr).
\end{align*}
For $1\le u\le N-1$, write
\[
 S_{N,u}(x)
 =
 \sum_{m=1}^{N-u}
 h\!\left(T^{\Omega(u+2m)}
          \bigl(T^{\Omega(u)}x\bigr)\right),
 \qquad
 y_u=T^{\Omega(u)}x.
\]
Then
\[
 S_{N,u}(x)
 =
 \sum_{\substack{u<v\le2N-u\\v\equiv u\pmod2}}
 h(T^{\Omega(v)}y_u).
\]

For $r\in\{0,1\}$, define the parity-restricted partial sums
\[
 P_r(L,y)
 =
 \sum_{\substack{1\le v\le L\\v\equiv r\pmod2}}
 h(T^{\Omega(v)}y).
\]
Let $r_u\in\{0,1\}$ satisfy $r_u\equiv u\pmod2$. Then
\[
 S_{N,u}(x)
 =
 P_{r_u}(2N-u,y_u)-P_{r_u}(u,y_u).
\]

Set
\[
 R_N
 =
 \max_{r\in\{0,1\}}
 \sup_{y\in X}
 \max_{\substack{L\in\mathbb Z\\0\le L\le2N}}
 |P_r(L,y)|.
\]
By \eqref{eq:parity}, we have $R_N=o(N)$. By
\[
 \begin{aligned}
 |S_{N,u}(x)|
 &\le
 |P_{r_u}(2N-u,y_u)|+|P_{r_u}(u,y_u)| \le2R_N,
 \end{aligned}
\]
we obtain that
\[
 \begin{aligned}
 \left|2\sum_{u=1}^{N-1}S_{N,u}(x)\right|
 &\le2\sum_{u=1}^{N-1}|S_{N,u}(x)|\le4(N-1)R_N
 =o(N^2),
 \end{aligned}
\]
uniformly in $x$. Adding the diagonal contribution, which is
$O_f(N)$, proves \eqref{eq_diffsquares}.
\end{proof}

\section{Proof of Theorem~\ref{var:thm:main}}
\label{sec:varieties}

We will use the following absolute-value minor-arc estimate.  Its proof
combines the estimates in \cite[Sections 7 and 8]{BHB} with the tails of
the major-arc expansion.
\begin{lemma}\label{var:lem:absolute}
Let $R\ge1$, $d\ge2$, and let $\mathbf A=(A_1,\ldots,A_R)$ be degree-$d$ integer polynomials
in $m$ variables, and suppose
\[
 \B(\mathbf A_d)>R(R+1)C_d.
\]
For a fixed bounded box $\mathcal U\subset\R^m$, define
\[
 E_N(\boldsymbol\theta)=
 \sum_{\mathbf z\in N\mathcal U\cap\Z^m}
 e(\boldsymbol\theta\cdot\mathbf A(\mathbf z)).
\]
For an integer $Q\ge1$, let
\begin{equation}\label{var:eq:full-arcs}
 \M_N^{(R)}(Q)=
 \bigcup_{\substack{q\le Q,\ \mathbf a\bmod q\\
                    \gcd(q,a_1,\ldots,a_R)=1}}
 \left\{\boldsymbol\theta\in\T^R:
 \max_j\left\|\theta_j-\frac{a_j}{q}\right\|
 \le QN^{-d}\right\}.
\end{equation}
Then
\begin{equation}\label{var:eq:absolute-tightness}
 \lim_{Q\to\infty}\limsup_{N\to\infty}
 N^{Rd-m}\int_{\T^R\setminus\M_N^{(R)}(Q)}
 |E_N(\boldsymbol\theta)|\dd\boldsymbol\theta=0.
\end{equation}
\end{lemma}

\begin{proof}
Write $B_0=\B(\mathbf A_d)$ and put
\[
 \rho=\frac{B_0}{R C_d}>R+1,
 \qquad \eta=\frac{1}{2R+4},
 \qquad L=N^\eta.
\]
Let $\widetilde\M_N$ denote the union in
\eqref{var:eq:full-arcs} with the real parameter $L$ in place of $Q$.
The argument of \cite[Section 7]{BHB}, specialized to $R$ equations of
the same degree $d$, bounds the integral of the absolute value and gives
some $\delta>0$ such that
\[
 \int_{\T^R\setminus\widetilde\M_N}
 |E_N(\boldsymbol\theta)|\dd\boldsymbol\theta
 \ll N^{m-Rd-\delta}.
\]
The extension to polynomials with lower-degree terms is recorded after
Theorem 1.2 of \cite{BHB}.

Define
\[
 S(q,\mathbf a)=
 \sum_{\mathbf z\bmod q}e\bigl(\mathbf a\cdot\mathbf A(\mathbf z)/q\bigr),
 \qquad
 I(\boldsymbol\gamma)=
 \int_{\mathcal U}e\bigl(\boldsymbol\gamma\cdot\mathbf A_d(\mathbf u)\bigr)
 \dd\mathbf u.
\]
In the equal-degree case, the parameter $s_d$ of \cite{BHB} is
$R C_d/B_0=1/\rho$.  Thus \cite[Lemmas 8.2 and 8.3]{BHB} give, for
$\gcd(q,a_1,\ldots,a_R)=1$ and any $\epsilon>0$,
\[
 q^{-m}|S(q,\mathbf a)|\ll_\epsilon q^{-\rho+\epsilon},
 \qquad
 |I(\boldsymbol\gamma)|
 \ll_\epsilon (1+\|\boldsymbol\gamma\|_\infty)^{-\rho+\epsilon}.
\]
All constants here and below may depend on the fixed system and box,
but are independent of $N$, $q$, $\mathbf a$, and $Q$.

For $\boldsymbol\theta=\mathbf a/q+\boldsymbol\beta$ in one of the
enlarged arcs, put $\boldsymbol\gamma=N^d\boldsymbol\beta$.
Splitting into residue classes modulo $q$ and comparing each resulting
sum with its integral gives
\begin{align*}
 E_N(\mathbf a/q+\boldsymbol\beta)
 &=q^{-m}S(q,\mathbf a)
   \int_{N\mathcal U}e\bigl(\boldsymbol\beta\cdot\mathbf A(\mathbf t)\bigr)
   \dd\mathbf t
   +O\bigl(qN^{m-1}(1+\|\boldsymbol\gamma\|_\infty)\bigr)\\
 &=N^m q^{-m}S(q,\mathbf a)I(\boldsymbol\gamma)
   +O\bigl(qN^{m-1}(1+\|\boldsymbol\gamma\|_\infty)\bigr).
\end{align*}
Indeed, uniformly for $\mathbf u\in\mathcal U$,
$N^{-d}\mathbf A(N\mathbf u)=\mathbf A_d(\mathbf u)+O(N^{-1})$.
The resulting change in the integral is
$O(N^{m-1}\|\boldsymbol\gamma\|_\infty)$, which is absorbed by the
displayed error because $q^{-m}|S(q,\mathbf a)|\le1$.
There are $O(L^{R+1})$ enlarged arcs, each of measure
$O(L^R N^{-Rd})$.  Consequently the integrated error, after
multiplication by $N^{Rd-m}$, is
\[
 O\bigl(N^{-1}L^{2R+3}\bigr)=O(N^{-\eta}).
\]

Fix $0<\epsilon<\rho-R-1$.  For fixed $Q$ and large $N$, a point of
$\widetilde\M_N\setminus\M_N^{(R)}(Q)$ lies in an enlarged arc for
which either $q>Q$ or $\|\boldsymbol\gamma\|_\infty>Q$.
The preceding expansion and the change of variables
$\boldsymbol\gamma=N^d\boldsymbol\beta$ therefore imply
\begin{align*}
 &\quad \limsup_{N\to\infty}N^{Rd-m}
 \int_{\T^R\setminus\M_N^{(R)}(Q)}
 |E_N(\boldsymbol\theta)|\dd\boldsymbol\theta\\
 &\ll_\epsilon
 \left(\sum_{q>Q}q^{R-\rho+\epsilon}\right)
 \int_{\R^R}(1+\|\boldsymbol\gamma\|_\infty)^{-\rho+\epsilon}
 \dd\boldsymbol\gamma\\
 &\qquad+
 \left(\sum_{q\ge1}q^{R-\rho+\epsilon}\right)
 \int_{\|\boldsymbol\gamma\|_\infty>Q}
 (1+\|\boldsymbol\gamma\|_\infty)^{-\rho+\epsilon}
 \dd\boldsymbol\gamma\\
 &\ll_\epsilon
 Q^{R+1-\rho+\epsilon}+Q^{R-\rho+\epsilon}.
\end{align*}
Both terms tend to zero as $Q\to\infty$, proving
\eqref{var:eq:absolute-tightness}.
\end{proof}

\begin{lemma}\label{var:lem:doubled-rank}
With the notation of Theorem~\ref{var:thm:main}, define
\[
 \mathbf Q(\mathbf u,\mathbf v)=
 (\mathbf F(\mathbf u),\mathbf F(\mathbf v),
   P(\mathbf u)-P(\mathbf v)).
\]
This is a system of $2r+1$ polynomials in $2n$ variables, and
\begin{equation}\label{var:eq:doubled-rank}
 \B(\mathbf Q_d)
 =\min\{\B(\mathbf G),\,2\B(\mathbf G,H)\}.
\end{equation}
\end{lemma}

\begin{proof}
Write $S=\Crit(\mathbf G)$ and $S'=\Crit(\mathbf G,H)$. The
Jacobian of $\mathbf Q_d$ has block form
\[
 \begin{pmatrix}
 J_{\mathbf G}(\mathbf u)&0\\
 0&J_{\mathbf G}(\mathbf v)\\
 \nabla H(\mathbf u)&-\nabla H(\mathbf v)
 \end{pmatrix}.
\]
If one of the first two blocks has deficient row rank, this matrix
has deficient row rank. If both blocks have full row rank, the matrix
has deficient row rank precisely when $\nabla H$ belongs to the
corresponding row span at both points. Thus
\begin{equation}\label{var:eq:critical-union}
 \Crit(\mathbf Q_d)
 =(S\times\C^n)\cup(\C^n\times S)\cup(S'\times S').
\end{equation}
Taking codimensions gives \eqref{var:eq:doubled-rank}. For $r=0$
the first two sets are empty and this is the rank-doubling
identity \eqref{eq:rankdoubling} for $H(\mathbf u)-H(\mathbf v)$.
\end{proof}

\begin{proposition}\label{var:prop:energy}
Under the rank hypothesis of Theorem~\ref{var:thm:main}, put
\[
 S_{V,P,N}(\alpha)=\sum_{\mathbf z\in V_N}e(\alpha P(\mathbf z)),
 \qquad h=n-rd.
\]
For the scalar arcs \eqref{eq:fixedarcs},  one has
\begin{equation}\label{var:eq:constrained-energy}
 \lim_{Q\to\infty}\limsup_{N\to\infty}
 N^{d-2h}\int_{\m_N(Q)}|S_{V,P,N}(\alpha)|^2\dd\alpha=0.
\end{equation}
\end{proposition}

\begin{proof}
Let
\[
\mathbf Q(\mathbf u,\mathbf v)
=
\bigl(
\mathbf F(\mathbf u),\mathbf F(\mathbf v),
P(\mathbf u)-P(\mathbf v)
\bigr).
\]
For $\boldsymbol\beta,\boldsymbol\gamma\in\mathbb T^r$
and $\alpha\in\mathbb T$, define
\[
E_{\mathbf Q,N}
(\boldsymbol\beta,\boldsymbol\gamma,\alpha)
=
\sum_{\mathbf u,\mathbf v\in N\mathcal R\cap\mathbb Z^n}
e\!\left(
\boldsymbol\beta\cdot\mathbf F(\mathbf u)
+\boldsymbol\gamma\cdot\mathbf F(\mathbf v)
+\alpha\bigl(P(\mathbf u)-P(\mathbf v)\bigr)
\right).
\]
Then for every $\alpha\in\T$,
\begin{equation}\label{var:eq:orthogonality}
 |S_{V,P,N}(\alpha)|^2
 =\int_{\T^{2r}}
 E_{\mathbf Q,N}(\boldsymbol\beta,\boldsymbol\gamma,\alpha)
 \dd\boldsymbol\beta\dd\boldsymbol\gamma.
\end{equation}
If a vector belongs to $\M_N^{(2r+1)}(Q)$, its last coordinate
belongs to $\M_N(Q)$. Therefore
\[
 \T^{2r}\times\m_N(Q)
 \subseteq\T^{2r+1}\setminus\M_N^{(2r+1)}(Q).
\]
Integrating \eqref{var:eq:orthogonality} and applying the triangle
inequality yields
\begin{equation}\label{var:eq:energy-comparison}
 \int_{\m_N(Q)}|S_{V,P,N}(\alpha)|^2\dd\alpha
 \le\int_{\T^{2r+1}\setminus\M_N^{(2r+1)}(Q)}
 |E_{\mathbf Q,N}(\boldsymbol\theta)|\dd\boldsymbol\theta.
\end{equation}
Use Lemma~\ref{var:lem:doubled-rank} and apply Lemma~\ref{var:lem:absolute} with $m=2n$ and $R=2r+1$. Since
$Rd-m=(2r+1)d-2n=d-2h$, this proves
\eqref{var:eq:constrained-energy}.
\end{proof}

\begin{proof}[Proof of Theorem~\ref{var:thm:main}]
Similar to Proposition \ref{prop:transfer}, we establish the more general conclusion
\begin{equation}\label{var:eq:rational-transfer}
 \lim_{N\to\infty} \frac1{|V_N|}\sum_{\mathbf z\in V_N}b(P(\mathbf z))=0
\end{equation}
for every bounded $b:\Z\to\C$ having vanishing rational means in
the sense of Definition~\ref{def:null}. Then Theorem~\ref{var:thm:main} follows by Corollary~\ref{lem:arithmetic}.

Choose a fixed integer $C\ge1$ with $|P(\mathbf z)|\le CN^d$ on
$N\mathcal R$, and set
\[
 B_N(\alpha)=\sum_{|t|\le CN^d}b(t)e(-t\alpha).
\]
Then by Parseval's formula, we have
\begin{equation}\label{var:eq:transfer-fourier}
 \sum_{\mathbf z\in V_N}b(P(\mathbf z))
 =\int_{\T}S_{V,P,N}(\alpha)B_N(\alpha)\dd\alpha,
 \qquad \|B_N\|_2\ll_b N^{d/2}.
\end{equation}
For every fixed $Q$, vanishing rational means and partial summation,
exactly as in \eqref{eq:maximalpartial}--\eqref{eq:Bmajor}, imply
\[
 \sup_{\alpha\in\M_N(Q)}|B_N(\alpha)|=o_Q(N^d).
\]
The estimates $|S_{V,P,N}|\le|V_N|$ and
$|\M_N(Q)|\ll_QN^{-d}$ show that the  major-arc
contribution in \eqref{var:eq:transfer-fourier}, after division by $|V_N|$, is $o_Q(1)$.
For the minor arcs, Cauchy--Schwarz and
\eqref{var:eq:base-count} give
\begin{align*}
 \frac1{|V_N|}\left|
 \int_{\m_N(Q)}S_{V,P,N}(\alpha)B_N(\alpha)\dd\alpha\right|
 &\ll_b
 \left(N^{d-2h}\int_{\m_N(Q)}|S_{V,P,N}(\alpha)|^2\dd\alpha
 \right)^{1/2}.
\end{align*}
Then Proposition~\ref{var:prop:energy}  implies
\eqref{var:eq:rational-transfer}, by first taking 
$N\to\infty$ and then letting $Q\to\infty$.
\end{proof}

\section{Proof of Theorem~\ref{diagvar:thm:main}}
\label{sec:diagonal-varieties}

Recall that $W_{d,N}(\alpha)=\sum_{t=1}^N e(\alpha t^d)$,
$d\ge2$, and the major arcs $\M_N(Q)$ and the minor arcs $ \m_N(Q)$ are defined by \eqref{eq:fixedarcs}.  Put $s=2K_*(d)$ and
\[
 M_s(N)=\int_{\T}|W_{d,N}(\alpha)|^s\dd\alpha,
 \qquad
 T_s(N,Q)=\int_{\m_N(Q)}|W_{d,N}(\alpha)|^s\dd\alpha.
\]
The estimates needed below are
\begin{equation}\label{diagvar:eq:scalar-input}
 M_s(N)\ll N^{s-d},\qquad
 \lim_{Q\to\infty}\limsup_{N\to\infty}N^{d-s}T_s(N,Q)=0.
\end{equation}
For $d=3$ the second assertion follows from
Corollary~\ref{cor:classicaldiagonal}, and for $d\ge4$ it follows
from Proposition~\ref{prop:modernenergy}. For $d=2$, use
Proposition~\ref{prop:Birchenergy} with
$P(t_1,t_2,t_3)=t_1^2+t_2^2+t_3^2$, whose Birch rank is $3>2$,
and $|S_{P,N}(\alpha)|^2=|W_{2,N}(\alpha)|^6$.
In all cases, the first assertion follows from the second: for one
sufficiently large fixed $Q$, the minor contribution is
$O(N^{s-d})$, while the trivial bound $|W_{d,N}|\le N$ and
$|\M_N(Q)|\ll_QN^{-d}$ give the same bound for the major contribution.

\begin{lemma}\label{diagvar:lem:blocks}
Let $C$ be an integer $R\times L$ matrix, with columns
$\mathbf c_1,\ldots,\mathbf c_L$. Suppose there are $s$ pairwise
disjoint sets $J_1,\ldots,J_s$ of column indices, each of size $R$,
such that every square matrix $C_{J_b}$ is invertible over $\Q$, $1\leq b\le s$.
Define
\[
 E_{C,N}(\boldsymbol\theta)
 =\prod_{j=1}^L W_{d,N}(\mathbf c_j\cdot\boldsymbol\theta).
\]
Then
\begin{equation}\label{diagvar:eq:full-lone}
 \int_{\T^R}|E_{C,N}(\boldsymbol\theta)|\dd\boldsymbol\theta
 \ll_C N^{L-Rd},
\end{equation}
and, for the full fixed arcs \eqref{var:eq:full-arcs},
\begin{equation}\label{diagvar:eq:joint-tightness}
 \lim_{Q\to\infty}\limsup_{N\to\infty}
 N^{Rd-L}\int_{\T^R\setminus\M_N^{(R)}(Q)}
 |E_{C,N}(\boldsymbol\theta)|\dd\boldsymbol\theta=0.
\end{equation}
\end{lemma}

\begin{proof}
For $1\le b\le s$, set
\[
 U_b(\boldsymbol\theta)
 =\prod_{j\in J_b}|W_{d,N}(\mathbf c_j\cdot\boldsymbol\theta)|.
\]
The integer matrix $C_{J_b}^{\mathsf T}$ induces a surjective
endomorphism of $\T^R$ and hence preserves normalized Haar measure.
In particular, its determinant need not equal $\pm1$. Therefore
\begin{equation}\label{diagvar:eq:block-moment}
 \int_{\T^R}U_b(\boldsymbol\theta)^s\dd\boldsymbol\theta
 =M_s(N)^R.
\end{equation}
Bound the $L-sR$ unselected Weyl factors by $N$. H\"older's
inequality with $s$ factors, followed by
\eqref{diagvar:eq:block-moment} and \eqref{diagvar:eq:scalar-input},
proves \eqref{diagvar:eq:full-lone}.

Fix a block $J_b=\{j_1,\ldots,j_R\}$ and define
\[
 L_b:\mathbb T^R\longrightarrow\mathbb T^R,\qquad
 L_b(\boldsymbol\theta)=C_{J_b}^{\,T}\boldsymbol\theta
 \pmod{\mathbb Z^R}.
\]
Then
\[
 U_b(\boldsymbol\theta)^s
 =\prod_{k=1}^R
 \left|W_{d,N}\bigl((L_b\boldsymbol\theta)_k\bigr)\right|^s.
\]

Notice that there is an increasing function $\Psi(Q)\ll_C(1+Q)^R$ with the
following property for every block: if all the $R$  arguments
$\mathbf c_j\cdot\boldsymbol\theta$ belong to $\M_N(Q)$, then
$\boldsymbol\theta\in\M_N^{(R)}(\Psi(Q))$. In $D_N(Q)=\T^R\setminus\M_N^{(R)}(\Psi(Q))$, every point  has at least one coordinate of $L_b\boldsymbol\theta$ in
$\mathfrak m_N(Q)$. Thus
\[
 \mathbf 1_{D_N(Q)}(\boldsymbol\theta)
 \le
 \sum_{\ell=1}^R
 \mathbf 1_{\mathfrak m_N(Q)}
 \bigl((L_b\boldsymbol\theta)_\ell\bigr).
\]

It follows that 
\begin{align*}
 \int_{D_N(Q)}U_b(\boldsymbol\theta)^s
 \,\dd\boldsymbol\theta
 &\le
 \sum_{\ell=1}^R
 \int_{\mathbb T^R}
 \mathbf 1_{\mathfrak m_N(Q)}
 \bigl((L_b\boldsymbol\theta)_\ell\bigr)
 \prod_{k=1}^R
 \left|W_{d,N}\bigl((L_b\boldsymbol\theta)_k\bigr)\right|^s
 \,\dd\boldsymbol\theta\\
 &=
 \sum_{\ell=1}^R
 \int_{\mathbb T^R}
 \mathbf 1_{\mathfrak m_N(Q)}(\alpha_\ell)
 \prod_{k=1}^R|W_{d,N}(\alpha_k)|^s
 \,\dd\boldsymbol\alpha\\
 &=
 \sum_{\ell=1}^R
 \left(
 \int_{\mathfrak m_N(Q)}|W_{d,N}(\alpha)|^s\,\dd\alpha
 \right)
 \left(
 \int_{\mathbb T}|W_{d,N}(\alpha)|^s\,\dd\alpha
 \right)^{R-1}\\
 &=R\,T_s(N,Q)M_s(N)^{R-1}.
\end{align*}

Applying H\"older's inequality yields
\begin{align*}
 \int_{D_N(Q)}|E_{C,N}|\dd\boldsymbol\theta
 &\le N^{L-sR}\prod_{b=1}^s
 \left(\int_{D_N(Q)}U_b^s\dd\boldsymbol\theta\right)^{1/s}\\
 &\le R N^{L-sR}T_s(N,Q)M_s(N)^{R-1}.
\end{align*}
After multiplication by $N^{Rd-L}$, the right side is
$O_C(N^{d-s}T_s(N,Q))$. It tends to zero in the iterated-limit
sense by \eqref{diagvar:eq:scalar-input}. Finally, the complements
of the fixed major arcs decrease as their parameter increases, so
convergence along the parameters $\Psi(Q)$ implies
\eqref{diagvar:eq:joint-tightness}.
\end{proof}

\begin{proof}[Proof of Theorem~\ref{diagvar:thm:main}]
Under the hypotheses of Theorem~\ref{diagvar:thm:main}, by \cite[Theorem~1.4 and Remark~1.6]{RomeYamagishi}, we have
\begin{equation}\label{diagvar:eq:count}
 |V_N|=c_VN^{n-rd}      +O(N^{n-rd-\delta}),\qquad c_V>0
\end{equation}
for some $\delta>0$.

Put $R=2r+1$, $\kappa=K_*(d)$, and $s=2\kappa$. Consider
\[
 \mathbf Q(\mathbf u,\mathbf v)
 =(\mathbf F(\mathbf u),\mathbf F(\mathbf v),
   P(\mathbf u)-P(\mathbf v)).
\]
This is a diagonal system in $2n$ variables. If $\mathbf a_j$ is
the $j$-th column of $A$, the coefficient matrix $C$ of $\mathbf Q(\mathbf u,\mathbf v)$ has columns
\[
 \mathbf u_j=(\mathbf a_j,0,\varepsilon_j),\qquad
 \mathbf v_j=(0,\mathbf a_j,-\varepsilon_j)
 \quad(1\le j\le n),
\]
where the zeros denote $r$-dimensional blocks.

Partition $\kappa R$ original indices into $\kappa$ disjoint groups,
each of the form $I\sqcup J$ with $|I|=r+1$ and $|J|=r$.
Each group gives two disjoint sets of $R$ columns of $C$:
\begin{equation}\label{diagvar:eq:paired-blocks}
 \{\mathbf u_i:i\in I\}\cup\{\mathbf v_j:j\in J\},
 \qquad
 \{\mathbf v_i:i\in I\}\cup\{\mathbf u_j:j\in J\}.
\end{equation}
Both sets are independent. For example, a linear functional
$(\boldsymbol\beta,\boldsymbol\gamma,\alpha)$ vanishing on the
first set has $(\boldsymbol\beta,\alpha)=0$ by invertibility of
$B_I$, and then $\boldsymbol\gamma=0$ by invertibility of $A_J$.
The second set is treated with the two blocks interchanged. Thus
$C$ has $2\kappa=s$ pairwise disjoint invertible $R$-column blocks.

Apply Lemma~\ref{diagvar:lem:blocks} to the exponential sum of
$\mathbf Q$. If
\[
 S_{V,P,N}(\alpha)=\sum_{\mathbf z\in V_N}e(\alpha P(\mathbf z)),
 \qquad h=n-rd,
\]
the orthogonality identity \eqref{var:eq:orthogonality} and equation \eqref{var:eq:energy-comparison} give
\[
 \int_{\m_N(Q)}|S_{V,P,N}(\alpha)|^2\dd\alpha
 \le\int_{\T^R\setminus\M_N^{(R)}(Q)}
 |E_{C,N}(\boldsymbol\theta)|\dd\boldsymbol\theta.
\]
Since $Rd-2n=d-2h$, it follows that
\begin{equation}\label{diagvar:eq:constrained-energy}
 \lim_{Q\to\infty}\limsup_{N\to\infty}
 N^{d-2h}\int_{\m_N(Q)}|S_{V,P,N}(\alpha)|^2\dd\alpha=0.
\end{equation}
Equation~\eqref{diagvar:eq:count} gives $|V_N|\sim c_VN^h$ with
$c_V>0$. The  argument in the proof of
Theorem~\ref{var:thm:main} now applies, using
\eqref{diagvar:eq:constrained-energy} in place of
\eqref{var:eq:constrained-energy}, and \eqref{diagvar:eq:main-limit} follows.
\end{proof}

\appendix

\section{Appendix. Uniformity for the sum of digits of the Omega function}\label{app:uniform-digits}

Write $\N_0=\{0,1,2,\ldots\}$. We first record two elementary
averaging facts, including the passage from bounded test sequences to
uniform estimates that is needed in Proposition~\ref{prop:diffsquares}.

\begin{lemma}\label{lem:bounded-testing}
Let $v_N\in\ell^1(\N_0)$ be real sequences. If
$\sum_{j\ge0}v_N(j)c(j)\to0$ for every bounded real sequence $c$,
then $\|v_N\|_1\to0$.
\end{lemma}
\begin{proof}
Testing on a single coordinate gives $v_N(j)\to0$ for each $j$.
Suppose that $\|v_N\|_1\ge\varepsilon>0$ along a subsequence.
Inductively choose indices $N_i$ from this subsequence and integers
$-1=M_0<M_1<M_2<\cdots$ such that
\[
 \sum_{j\le M_{i-1}}|v_{N_i}(j)|<\frac\varepsilon8,
 \qquad
 \sum_{j>M_i}|v_{N_i}(j)|<\frac\varepsilon8.
\]
The first choice is possible by coordinate convergence and the second
by summability. Define
$c(j)=\operatorname{sgn}(v_{N_i}(j))$ on
$M_{i-1}<j\le M_i$. The contribution of this block to
$\sum_jv_{N_i}(j)c(j)$ is at least $3\varepsilon/4$, and the absolute
contribution of its complement is less than $\varepsilon/4$.
This contradicts convergence to zero for the fixed bounded sequence $c$.
\end{proof}

\begin{lemma}\label{lem:probability-uniform}
Let $\mu_N$ be probability mass functions on $\N_0$, extended by zero
to $\Z$, and suppose that
\begin{equation}\label{eq:probability-shift}
 \delta_N:=\sum_{j\in\Z}|\mu_N(j-1)-\mu_N(j)|\longrightarrow0.
\end{equation}
For every uniquely ergodic system $(X,\nu,T)$ and $f\in C(X)$,
\[
 \sup_{y\in X}\left|
 \sum_{j\ge0}\mu_N(j)f(T^jy)-\av f\right|\longrightarrow0.
\]
\end{lemma}
\begin{proof}
Put $A_Nf(y)=\sum_{j\ge0}\mu_N(j)f(T^jy)$.
Reindexing the sum gives
\[
 \|A_N(f\circ T)-A_Nf\|_\infty
 \le\|f\|_\infty\delta_N.
\]
For $H\ge1$, let $f_H=H^{-1}\sum_{i=0}^{H-1}f\circ T^i$.
Telescoping and the fact that $A_N$ is a contraction give
\[
 \|A_Nf-\av f\|_\infty
 \le\frac{H-1}{2}\|f\|_\infty\delta_N
      +\|f_H-\av f\|_\infty.
\]
Unique ergodicity implies uniform convergence of $f_H$ to $\av f$.
Let $N\to\infty$ for fixed $H$, and then let $H\to\infty$.
\end{proof}

\begin{lemma}\label{lem:omega-ap-shift}
For fixed integers $m\ge1$ and $0\le a<m$, set
\[
 \rho_{N;m,a}(j)=\frac1N
   \#\{1\le n\le N:\Omega(mn+a)=j\},\qquad j\ge0.
\]
Extending these functions by zero to negative integers, one has
\begin{equation}\label{eq:omega-ap-shift}
 \sum_{j\in\Z}
 |\rho_{N;m,a}(j-1)-\rho_{N;m,a}(j)|\longrightarrow0.
\end{equation}
\end{lemma}
\begin{proof}
We apply \cite[Proposition 3.2]{BR}, which requires a finitely
generated multiplicative topological system, without any ergodicity
assumption. Let $c:\N_0\to[-1,1]$ be arbitrary and extend it to $\Z$.
On
\[
 Y=[-1,1]^{\Z}\times(\Z/m\Z)
\]
let $\tau$ be the bilateral shift $(\tau u)_j=u_{j+1}$ and define
\[
 S_n(u,r)=(\tau^{\Omega(n)}u,nr),\qquad n\ge1.
\]
These continuous maps form a multiplicative action. There are only
finitely many maps $S_p$ as $p$ ranges over the primes, since $S_p$
depends only on $p\bmod m$. For primes $p\equiv1\pmod m$ the map is
$R=(\tau,\mathrm{id})$, and
$\sum_{p\equiv1\ (m)}p^{-1}=\infty$ by Dirichlet's theorem.
Take $u_j=c(j)$, $y=(u,\overline1)$ and
$g(v,r)=v_{-1}\mathbf1_{\{r=\overline a\}}$.
The cited proposition says that the averages of $g(RS_ny)$ and
$g(R^2S_ny)$ have difference tending to zero. In this system this is
\[
 \frac1M\sum_{\substack{n\le M\\n\equiv a\ (m)}}
       \bigl(c(\Omega(n)+1)-c(\Omega(n))\bigr)\longrightarrow0.
\]
Set $M=mN+a$ and discard the possible initial term $n=a$.
After multiplying by $M/N$, it follows that
\[
 \sum_{j\ge0}
  \bigl(\rho_{N;m,a}(j-1)-\rho_{N;m,a}(j)\bigr)c(j)
 =\frac1N\sum_{n=1}^N
  \bigl(c(\Omega(mn+a)+1)-c(\Omega(mn+a))\bigr)
 \longrightarrow0.
\]
Scaling permits any bounded real $c$, so
Lemma~\ref{lem:bounded-testing} proves \eqref{eq:omega-ap-shift}.
\end{proof}

\begin{proposition}\label{prop:digit-shift-transfer}
Let $b\ge2$ be an integer, and let $\mu_N$ be probability mass
functions satisfying \eqref{eq:probability-shift}. For every uniquely
ergodic system $(X,\nu,T)$ and $f\in C(X)$,
\[
 \sup_{x\in X}\left|
   \sum_{j\ge0}\mu_N(j)f(T^{s_b(j)}x)-\av f
 \right|\longrightarrow0.
\]
\end{proposition}
\begin{proof}
For $L\ge1$, let $\eta_L$ be the distribution of the sum of $L$
independent uniform digits in $\{0,\ldots,b-1\}$. Define
\[
 D_Lf(y)=\frac1{b^L}\sum_{r=0}^{b^L-1}f(T^{s_b(r)}y)
        =\sum_j\eta_L(j)f(T^jy).
\]
We first show that $\|\eta_L(\,\cdot-1)-\eta_L\|_1\to0$.
The Fourier transform of one digit is
$\phi_b(t)=b^{-1}\sum_{r=0}^{b-1}e(rt)$. For $|t|\le1/2$,
\begin{align*}
 1-|\phi_b(t)|^2
 &=\frac2{b^2}\sum_{0\le r<s<b}
       \bigl(1-\cos(2\pi(s-r)t)\bigr)\\
 &\ge\frac{2(b-1)}{b^2}(1-\cos(2\pi t))
 \ge\frac{16(b-1)}{b^2}t^2.
\end{align*}
Consequently $|\phi_b(t)|^{2L}\le\exp(-c_bLt^2)$ for a constant
$c_b>0$. Parseval's identity gives
\[
 \|\eta_L(\,\cdot-1)-\eta_L\|_2^2
 =\int_{-1/2}^{1/2}|e(t)-1|^2|\phi_b(t)|^{2L}\dd t
 \ll_b L^{-3/2}.
\]
The difference is supported on at most $(b-1)L+2$ integers, so
Cauchy--Schwarz gives
$\|\eta_L(\,\cdot-1)-\eta_L\|_1\ll_b L^{-1/4}$.
Lemma~\ref{lem:probability-uniform} now implies
\begin{equation}\label{eq:digit-block-uniform}
 \|D_Lf-\av f\|_\infty\longrightarrow0.
\end{equation}

Fix $L$ and put $B=b^L$. Average $\mu_N$ on each consecutive block
of length $B$ by setting
\[
 w_{N,v}=\sum_{r=0}^{B-1}\mu_N(Bv+r),\qquad
 \overline\mu_N(Bv+r)=\frac{w_{N,v}}B
 \quad(v\ge0,\ 0\le r<B).
\]
For each block, comparison with the mean and telescoping yield
\[
 \sum_{r=0}^{B-1}
  \left|\mu_N(Bv+r)-\frac{w_{N,v}}B\right|
 \le B\sum_{r=0}^{B-2}|\mu_N(Bv+r+1)-\mu_N(Bv+r)|.
\]
Summing gives $\|\mu_N-\overline\mu_N\|_1\le B\delta_N$.
Since $s_b(Bv+r)=s_b(v)+s_b(r)$ and $\sum_vw_{N,v}=1$,
\[
 \sum_{j\ge0}\overline\mu_N(j)f(T^{s_b(j)}x)
 =\sum_{v\ge0}w_{N,v}D_Lf(T^{s_b(v)}x).
\]
It follows that
\[
 \sup_{x\in X}\left|
   \sum_{j\ge0}\mu_N(j)f(T^{s_b(j)}x)-\av f\right|
 \le b^L\|f\|_\infty\delta_N
      +\|D_Lf-\av f\|_\infty.
\]
First let $N\to\infty$ with $L$ fixed and then let $L\to\infty$,
using \eqref{eq:digit-block-uniform}.
\end{proof}

\begin{corollary}\label{cor:digit-omega-ap}
For every integer base $b\ge2$, every $m\ge1$ and $0\le a<m$,
every uniquely ergodic system $(X,\nu,T)$, and every $f\in C(X)$,
\[
 \lim_{N\to\infty}\sup_{x\in X}
 \left|\frac1N\sum_{n=1}^N
    f(T^{s_b(\Omega(mn+a))}x)-\av f\right|=0.
\]
In particular, for fixed $x$, the bounded even sequence
$j\mapsto f(T^{s_b(\Omega(j))}x)-\av f$ has vanishing rational
means. Therefore Proposition~\ref{prop:transfer} proves the
digit-sum variant of \eqref{eq:generalmain} under the hypotheses
of Theorem~\ref{thm:Birch}.
\end{corollary}
\begin{proof}
Apply Proposition~\ref{prop:digit-shift-transfer} to
$\mu_N=\rho_{N;m,a}$, using Lemma~\ref{lem:omega-ap-shift}. Lemma~\ref{lem_ap} then gives the asserted
vanishing rational means.
\end{proof}

\section*{Acknowledgments}

We thank  Zhi Qi for some interesting discussions and for sharing the preprint \cite{QZ} with us.  During the preparation of this manuscript, ChatGPT was used to explore some of the proof strategies and calculations. The authors take full responsibility for human verifications and the content.

\end{document}